\documentclass[12pt,a4paper,english]{article}
\usepackage{vmargin}
\usepackage[utf8]{inputenc}
\usepackage[english]{babel}
\usepackage{hyperref}

\usepackage[dvipsnames]{xcolor}
\usepackage{amsthm, amsmath, amsfonts, amssymb, amscd, mathrsfs, comment}
\usepackage{mathpazo}
\usepackage[scaled=.95]{helvet}
\usepackage{courier}
\usepackage{ragged2e}
\usepackage{tikz-cd} 

\usepackage{hyphenat}
\newcommand{\gen}[1]{\langle #1\rangle}

\newcommand\ke{\mathop{\rm Ker}}
\newcommand\im{\mathop{\rm Im}}
\newcommand\Hom{\mathop{\rm Hom}}
\newcommand\aut{\mathop{\rm Aut}}
\newcommand\End{\mathop{\rm End}}

\usepackage{tikz}

\makeatletter
\newbox\xrat@below
\newbox\xrat@above
\newcommand{\xrightarrowtail}[2][]{%
  \setbox\xrat@below=\hbox{\ensuremath{\scriptstyle #1}}%
  \setbox\xrat@above=\hbox{\ensuremath{\scriptstyle #2}}%
  \pgfmathsetlengthmacro{\xrat@len}{max(\wd\xrat@below,\wd\xrat@above)+3.6em}%
  \mathrel{\tikz [>->,baseline=-.75ex]
                 \draw (0,0) -- node[below=-2pt] {\box\xrat@below}
                                node[above=-2pt] {\box\xrat@above}
                       (\xrat@len,0) ;}}
\makeatother

\makeatletter
\newbox\xrat@below
\newbox\xrat@above
\newcommand{\xtwoheadrightarrow}[2][]{%
  \setbox\xrat@below=\hbox{\ensuremath{\scriptstyle #1}}%
  \setbox\xrat@above=\hbox{\ensuremath{\scriptstyle #2}}%
  \pgfmathsetlengthmacro{\xrat@len}{max(\wd\xrat@below,\wd\xrat@above)+3.6em}%
  \mathrel{\tikz [->>,baseline=-.75ex]
                 \draw (0,0) -- node[below=-2pt] {\box\xrat@below}
                                node[above=-2pt] {\box\xrat@above}
                       (\xrat@len,0) ;}}

\newtheorem{theorem}{Theorem}

\newtheorem{lemma}[theorem]{Lemma}
\newtheorem{example}[theorem]{Example}
\newtheorem{cor}[theorem]{Corollary}
\newtheorem{prop}[theorem]{Proposition}

\newtheorem{prob}{Problem}

\newtheorem*{thmA}{Theorem A}
\newtheorem*{thmB}{Corollary B}
\newtheorem*{thmC}{Theorem C}
\newtheorem*{thmD}{Theorem D}
\newtheorem*{thmE}{Theorem E}

\numberwithin{theorem}{section}

\numberwithin{prob}{section}

\begin{document}
\thispagestyle{empty}
\centerline{}
\textcolor{Blue}{{\LARGE\textsc{\textbf{\begin{center}A determinant for automorphisms of groups
\footnote{The author is member of GNSAGA-INdAM and ADV-AGTA. The work is carried under the co-financing of the European Union - "FSE-REACT-EU, PON Research and Innovation
2014-2020, DM 1062/202".}
\end{center}}}}}

\vspace{0.6cm}

\noindent\hfil\textcolor{Maroon}{{\Large\textsc{\textbf{Mattia Brescia}}}}\hfil

\vspace{1cm}
\justify
\textcolor{Blue}{\textbf{\centerline{Abstract}}}
\begin{changemargin}{1cm}{1cm}
Let $H$ and $K$ be groups. In this paper we introduce a concept of determinant for endomorphisms of $H\times K$ and some concepts of incompatibility for group pairs as a measure of how much $H$ and $K$ are far from being isomorphic. With the aid of the tools developed, we give a characterisation of invertible automorphisms of $H\times K$ by means of their determinants and an explicit description of $\aut(H\times K)$ as a group of $2$-by-$2$ matrices, in case $H$ or $K$ belong to some relevant classes of groups. Many theoretical and practical applications of the determinants will be presented, together with examples and an analysis on some computational advantages of the determinants. In particular, we give a matrix characterization of the whole automorphism group of some relevant classes of groups.
\end{changemargin}



\vskip 2\baselineskip

\section{Introduction}

Despite the many recent advances, the study of automorphisms of groups is still at an early stage and several are the unsolved theoretical problems concerning even apparently simple questions. Moreover, a knowledge about the general behaviour of automorphisms has been proven to be crucial in many areas of mathematics and, in recent years, increasing emphasis has been put on infinite groups and their automorphisms in different fields. For instance, in the context of (skew) braces and the Yang-Baxter equation, automorphisms of (non-)abelian groups appear in a very natural way as an action of multiplicative group over the additive group. Also, direct products and automorphisms of groups are one of the most common way to construct examples of braces (see, among the others, the construction of a simple left brace in \cite{Bach} and more in general \cite{CeJeOk}). Another area is that of group-based cryptography, where infinite non-abelian groups have been employed to produce secure cryptographic protocols which make use of some prescribed automorphisms. As in the case of the twisted conjugacy problem, a knowledge of the automorphisms of a group can prove, loosely speaking, secure or insecure one specific protocol (about this and other problems, see, for instance, \cite{MyShUs} and \cite{Sh}).

The study of automorphism of direct products of groups has a long story and a matrix definition of endomorphisms of a direct product of finitely many isomorphic non-abelian groups dates back to \cite{John}, in the context of groups satisfying Poincaré duality. Successively, Bidwell, Curran and McCaughan in \cite{BidCurMcC}, and then Bidwell in \cite{Bid}, studied the automorphism group of the direct product of finitely many finite groups. In particular, they proved that if two groups $H$ and $K$ share no isomorphic non-trivial direct factor, then $\aut(H\times K)$ can be fully described as a particular set of $2$-by-$2$ matrices, which we will call $\mathcal{A}_{H,K}$ (for its definition, see Section \ref{SecPre}). This brings a clear computational advantage and provides many useful applications (among the others, see, for instance, the very recent \cite{CCCo}, \cite{King},
%\cite{Tiep}, 
\cite{Send1} and \cite{Send2}). On the other hand, such a good description was still missing for infinite groups. Moreover, in \cite{BidCurMcC} it is proven, under certain conditions, that $\mathcal{A}_{H,K}$ is a subset of $\aut(H\times K)$, but nothing is said about its group structure, except for the case when it coincides with the whole automorphism group. One of the purposes of this paper is to provide the theoretical background to complete the study of $\mathcal{A}_{H,K}$ for finite groups $H$ and $K$ and to extend the results to infinite groups. This will lead to a matrix characterization of the automorphism group of some relevant classes of groups, together with efficient algorithms checking the bijectivity of endomorphisms of (not necessarily finite) groups.

\subsection{A roadmap}
\medskip

Here we summarize the content of the paper, with a focus on the main results. Most of our notation is standard and can be found in \cite{RobCour}.

After the preliminaries, in Section \ref{SecDet} we define the determinants for some endomorphisms of $H\times K$ and, in particular, for every element of $\mathcal{A}_{H,K}$. This allows to prove, by means of Theorem \ref{propinv}, that an automorphism $\varphi$ of $H\times K$ is invertible if and only if one of its determinants is definable and invertible. In this case, $\varphi^{-1}$ is given an explicit formulation in terms of the components of $\varphi$ seen as a $2$-by-$2$ matrix. Although the determinants are not generally definable on every automorphism, the fact that this is always possible when $\varphi$ belongs to $\mathcal{A}_{H,K}$ makes it possible, by means of Lemma \ref{detiff}, to talk about \textit{the} determinant of every element of $\mathcal{A}_{H,K}$. Section \ref{SecDet} is concluded with a sketch of the theory for the case of direct products of $n$ groups for $n\geq2$, defining determinants in the general case, as well as stating results which can be derived from the case $n=2$. One of the main results of the section is the following characterization of invertible endomorphisms by means of the determinants (see Theorem \ref{propinvGen}).

\begin{thmA}
Let  $H_1,\ldots, H_n$ be groups and let $\varphi$ be an endomorphism of $H_1\times\ldots\times H_n$. Then $\varphi$ is invertible if and only if a determinant of $\varphi$ can be defined and and is invertible.
\end{thmA}

From this we will get immediately the following characterization of every invertible endomorphisms belonging to a specific subset of endomorphisms (see Corollary \ref{CorDetGen}).

\begin{thmB}
Let $H_1,\ldots, H_n$ be groups and let $\varphi$ be an element of $\mathcal{A}_{H_1,\ldots, H_n}$. Then $\varphi$ is invertible if and only $\det\varphi$ is invertible.
%Qui ho tolto "can be defined and", che era scritto non capisco perché
\end{thmB}

In Section \ref{SecPairs}, we define particular pairs of groups, which will provide a general framework where our tools can be implemented, together with general results about $\mathcal{A}_{H,K}$. Incompatibility, as we call it, may in fact be seen as some way of telling how far two groups are from being isomorphic. The main result of the section may be summed up as follows.

\begin{thmC}
Let $H$ and $K$ be groups. If
\begin{itemize}
\item $H$ and $K$ are incompatible and they both satisfy the minimal condition on normal subgroups (Theorem \ref{incompatible}) or
\item $H$ and $K$ are totally incompatible (Proposition \ref{htotincompatible}),
\end{itemize}
then $\aut (H\times K)\simeq\mathcal{A}_{H,K}$.
\end{thmC}

Indeed, group pairs and determinants are used in Section \ref{SecAppl} to find some specific classes of groups for which the whole $\aut(H\times K)$ is isomorphic with $\mathcal{A}_{H,K}$. Here we give some extensions of the results in \cite{BidCurMcC} to the infinite case, we answer to their open issue about $\mathcal{A}_{H,K}$ being a subgroup of $\aut(H\times K)$ and sharpen some of their results, also in the finite case, under the additional (and obvious) condition that $H$ and $K$ share no non-trivial direct factor.

\begin{thmD}
Let $H$ and $K$ be groups satisfying one of the following properties:
\begin{itemize}
\item $H$ and $K$ satisfy both the maximal and the minimal condition on normal subgroups (Theorem \ref{ThCompAut});
\item $H$ and $K$ are abelian and one of them is divisible-by-bounded-by-free (Theorem \ref{PropAb});
\item one between $H$ and $K$ is a stem group (Theorem \ref{stem}).
\end{itemize}
Then $\aut (H\times K)\simeq\mathcal{A}_{H,K}$ if and only if they have no common non-trivial direct factor.
\end{thmD}

Notice that in particular for these group pairs the determinant can be defined over all the automorphisms, so that, as shown in the last subsection of Section \ref{SecAppl}, the computational complexity of determining the inverse of an automorphism of a direct products of not necessarily finite groups can be drastically reduced, provided that the groups satisfy some good hypotheses. This, together with the explicit formula for inverse automorphisms, has an implementation in computer algebra systems as a natural application. The main result is the following.

\begin{thmE} Let $H$ and $K$ be groups and let $H$ be finite of order $m$. Let $\varphi$ be an endomorphism of $H\times K$ and let $\pi$ and $\iota$ be the canonical projection of $H\times K$ over $K$ and the immersion of $K$ into $H\times K$, respectively. If $\pi\varphi\iota$ is invertible, then $\varphi$ can be checked to be bijective in $\binom{m}{2}$ steps.

In particular, the hypothesis on $\pi\varphi\iota$ can be dropped if $H$ or $K$ satisfy one of the properties in Theorem C.
\end{thmE}

One of the strengths of this result is clearly that one gets a method to calculate the invertibility (and more) of a $\varphi\in\End(H\times K)$ with finitely many steps, even if one between $H$ and $K$ is an arbitrary infinite group.

In Section \ref{SecEx}, we implement the determinants to give several explicit examples, which have a twofold purpose: on one hand, they are meant to show how the theory developed so far can be employed; on the other hand, they point out that the hypotheses of many of our results cannot be removed or loosened.

We conclude the paper with a section about several possible future developments of the theory produced so far, both in a practical and theoretical direction. In particular, we stress the algorithmic implementation of determinants and give more definitions of group pairs which will be interesting to investigate.

\section{Preliminaries}\label{SecPre}

We start giving some easy, general results for direct products of monoids. In the following, if $M$ is a monoid and $f$ and $g$ are functions whose images are subsets of $M$, we will say that $[\im f,\im g]=1$ meaning that $\im f$ and $\im g$ commute elementwise. Let $H$ and $K$ be monoids and take into account the following set which is a slight generalization of that defined in \cite{John} and in \cite{BidCurMcC} $$\mathcal{M}_{H,K}=\left\{\left(\begin{array}{cc}\alpha &\beta \\ \gamma & \delta\end{array}\right):\begin{array}{lll} \alpha\in \End H, &\beta\in \Hom(K,H), &[\im\alpha,\im\beta]=1 \\
\gamma\in \Hom(H,K), &\delta\in \End K, &[\im\gamma,\im\delta]=1
\end{array}
\right\}$$

Endowed with row-by-column multiplication $``\cdot"$, where multiplying by coordinates means composing functions and a sum of functions is the usual pointwise sum, $(M,\cdot)$ is a monoid, which is isomorphic with $\End(H\times K)$ by the following Proposition whose easy proof can be retrieved from \cite{BidCurMcC} or in more general terms from \cite{John}.

\begin{prop}\label{end}
Let $H$ and $K$ be monoids. Then $\End(H\times K)\simeq \mathcal{M}_{H,K}$.
\end{prop}
\begin{comment}
\begin{proof}
Let $G=H\times K$. Let $\iota_H$ and $\iota_K$ be the inclusion maps from $H$ and $K$ in $G$, respectively, and let $\pi_H$ and $\pi_K$ be the projection of $G$ over, respectively, $H$ and $K$. For any endomorphism $\varphi$ of $G$, define $\alpha_\varphi=\pi_H\varphi\iota_H$, $\beta_\varphi=\pi_H\varphi\iota_K$, $\gamma_\varphi=\pi_K\varphi\iota_H$ and $\delta_\varphi=\pi_K\varphi\iota_K$. Since $$(\alpha_\varphi(h)\beta_\varphi(k),\gamma_\varphi(h)\delta_\varphi(k))=\varphi(h,1)\varphi(1,k)=$$ $$=\varphi(1,k)\varphi(h,1)=(\beta_\varphi(k)\alpha_\varphi(h),\delta_\varphi(k)\gamma_\varphi(h))$$ for any $h\in H$ and $k\in K$, we get that $\left(\begin{array}{cc}\alpha_\varphi &\beta_\varphi \\ \gamma_\varphi & \delta_\varphi\end{array}\right)\in \mathcal{M}_{H,K}$ and then we can define the monomorphism $\tau:\varphi\in\End G\,\mapsto \left(\begin{array}{cc}\alpha_\varphi &\beta_\varphi \\ \gamma_\varphi & \delta_\varphi\end{array}\right)\in \mathcal{M}_{H,K}$.

As one can easily see that the map $$\left(\begin{array}{cc}\alpha &\beta \\ \gamma & \delta\end{array}\right)\in \mathcal{M}_{H,K}\mapsto (\varphi:(h,k)\in G\mapsto (\alpha(h)\beta(k),\gamma(h)\delta(k)))\in\text{End}\, G$$ is the inverse of $\tau$, the thesis follows.
\end{proof}
\end{comment}

\medskip

Let $H$ and $K$ be two monoids and let $\varphi$ be an automorphism of $H\times K$. Making use of Proposition \ref{end}, in the following we will frequently write $$\varphi=\left(\begin{array}{cc}\alpha &\beta \\ \gamma & \delta\end{array}\right)\quad\text{and}\quad\varphi^{-1}=\left(\begin{array}{cc}\alpha' &\beta' \\ \gamma' & \delta'\end{array}\right)$$ where the components satisfy the properties inherited by $\mathcal{M}_{H,K}$.

\smallskip

With abuse of notation, throughout the paper we will regard $\aut(H\times K)$ as a subset of $\mathcal{M}_{H,K}$ for any two monoids $H$ and $K$.

From Proposition \ref{end}, we get the following well-known result, which is also easy to prove directly.

\begin{cor}\label{cor1}
Let $H$ and $K$ be monoids such that $\Hom(H,K)=0=\Hom(K,H)$. Then $\aut (H\times K)\simeq\aut H\times\aut K$.
\end{cor}

In this case, $\aut (H\times K)$ can be regarded as a particular diagonal subgroup of $\mathcal{M}_{H,K}$. Taking this result as a further motivation, one of our goals will be that of looking for conditions on $H$ and $K$ which will give a "good" description of $\aut (H\times K)$ in terms of its matrix representation. We begin inspecting general properties of automorphisms of direct products of monoids. The following is straightforward.

\begin{lemma}\label{cases}
Let $H$ and $K$ be monoids and let $\varphi$ be an automorphism of $H\times K$. Then the following relations hold
 \begin{equation*}
        \begin{matrix}
(1.1)& \alpha \alpha' + \beta \gamma' = 1
&\qquad&
(2.1)& \alpha' \alpha + \beta' \gamma = 1
\\
(1.2)& \alpha \beta' + \beta \delta' = 0
&\qquad&
(2.2)& \alpha' \beta + \beta' \delta = 0
\\
(1.3)& \gamma \alpha' + \delta \gamma' = 0
&\qquad&
(2.3)& \gamma' \alpha + \delta' \gamma = 0
\\
(1.4)& \gamma \beta' + \delta \delta' = 1
&\qquad&
(2.4)& \gamma' \beta + \delta' \delta = 1
    \end{matrix}
  \end{equation*}
\end{lemma}

Recall that an endomorphism of a monoid $M$ is said to be \textit{normal} if it commutes with every inner automorphism of $M$. In particular, if $\varphi$ is a normal endomorphism of $M$, then $\im\varphi^n$ is a normal submonoid of $M$ for any non-negative integer $n$.

\begin{lemma}\label{norm}
Let $H$ and $K$ be monoids and let $\varphi$ be an automorphism of $H\times K$. Then
\begin{itemize}
\item[(1)] If $\alpha$ is surjective, $\im\beta\leq Z(H)$. Similarly, if $\delta$ is surjective, $\im\gamma\leq Z(K)$;
\item[(2)] $\alpha\alpha'$, $\beta\gamma'$ and $\beta'\gamma$ are normal endomorphisms of $H$;
\item[(3)] $\delta\delta'$, $\gamma\beta'$ and $\beta\gamma'$ are normal endomorphisms of $K$.
\end{itemize}
\end{lemma}
\begin{proof}
(1) follows immediately from the commuting relations of $\mathcal{M}_{H,K}$.

(2) and (3) are completely symmetrical and it will be enough for us to show that $\alpha\alpha'$, $\beta\gamma'$ and $\beta'\gamma$ are normal endomorphisms of $H$. To this aim, let $h$ be an invertible element of $H$ and $x$ be an element in $\im\alpha$. Then, by Lemma \ref{cases} (1.1) we may write $h=\alpha\alpha'(h)\beta\gamma'(h)$. Since $h$ is invertible, $[\im\alpha,\im\beta]=1$ and clearly $\im\alpha\alpha'\subseteq\im\alpha$ and $\im\beta\gamma'\subseteq\im\beta$, we have that $\alpha\alpha'(h)$ and $\beta\gamma'(h)$ are invertible, so that we may consider the element $x^{\beta\gamma'(h)}$, which equals $x$, again because $[\im\alpha,\im\beta]=1$. This shows that $x^h=x^{\alpha\alpha'(h)}$. This holds for every element of $\im\alpha$, so that, if we take $m$ to be an element of $H$, we get that $$\alpha\alpha'(m^{h})=\alpha\alpha'(h^{-1})\alpha\alpha'(m)\alpha\alpha'(h)=(\alpha\alpha'(m))^{\alpha\alpha'(h)}=(\alpha\alpha'(m))^h.$$ Hence $\alpha\alpha'$ is a normal endomorphism of $H$.

Let now $y$ be an element of $\im \beta\gamma'$. For exactly the same reasons as before, we have that $y^h=y^{\beta\gamma'(h)}$. Since this holds for every element of $\im \beta\gamma'$, if $n$ is any element of $H$ we get $$\beta\gamma'(n^{h})=\beta\gamma'(h^{-1})\beta\gamma'(n)\beta\gamma'(h)=(\beta\gamma'(n))^{\beta\gamma'(h)}=(\beta\gamma'(n))^h.$$ Then, also $\beta\gamma'$ is a normal endomorphism of $H$. In a totally similar way, using Lemma \ref{cases} (2.1) we can show that $\beta'\gamma$ is normal, and the proof is concluded.
\end{proof}

\section{The determinants}\label{SecDet}

To go on with our study we need inverses of elements, so we are now on switching to groups. Let $H$ and $K$ be groups and let $\varphi$ be an element of $\End(H\times K)$. If $\delta$ is invertible, we define $\det_H\varphi=\alpha-\beta\delta^{-1}\gamma$ to be the $H$-\textit{determinant} of $\varphi$. In the same way, if $\alpha$ is invertible, we define $\det_K\varphi=\delta-\gamma\alpha^{-1}\beta$, the $K$-\textit{determinant} of $\varphi$. In general, $\det_H\varphi$ or $\det_K\varphi$ need not be homomorphisms.
%Tipo prendendo S3xZ2.

If one between the $H$- and the $K$-determinant of $\varphi$ can be defined and is invertible, it turns out that $\varphi$ is an automorphism of $H\times K$ and that $\varphi^{-1}$ can be given explicitly in terms of $\alpha$, $\beta$, $\gamma$ and $\delta$. Moreover, if $\det_H\varphi$, say, is invertible, then it is also an automorphism.

\begin{prop}\label{invgen}
Let $H$ and $K$ be groups and let $\varphi$ be an endomorphism of $H\times K$. Then the following hold:
\begin{itemize}
\item[(1)] If $\delta$ is invertible and $\Delta_H=\det_H\varphi$ is invertible, then $$\left(\begin{array}{cc}\Delta_H^{-1} &-\Delta_H^{-1}\beta\delta^{-1} \\ -\delta^{-1}\gamma\Delta_H^{-1} & (1+\delta^{-1}\gamma\Delta_H^{-1}\beta)\delta^{-1}\end{array}\right)$$ is the inverse of $\varphi$. Moreover, $\det_K\varphi^{-1}=\delta^{-1}$ and $\det_K\varphi$ is a homomorphism.
\item[(2)] If $\alpha$ is invertible and $\Delta_K=\det_K\varphi$ is invertible, then $$\left(\begin{array}{cc}(1+\alpha^{-1}\beta\Delta_K^{-1}\gamma)\alpha^{-1} &-\alpha^{-1}\beta\Delta_K^{-1} \\ -\Delta_K^{-1}\gamma\alpha^{-1} & \Delta_K^{-1}\end{array}\right)$$ is the inverse of $\varphi$. Moreover, $\det_H\varphi^{-1}=\alpha^{-1}$ and $\det_H\varphi$ is a homomorphism.
\end{itemize}
\end{prop}
\begin{proof}
We will prove the thesis only in case $\delta$ is invertible, as the other case is completely symmetric.

Let $\Delta=\det_H\varphi$. Notice first that since $\delta$ is bijective, $\im\gamma$ is a central subgroup of $K$.  Moreover, from $\Delta^{-1}(\alpha-\beta\delta^{-1}\gamma)=1$ one get the following relations

$$1+\beta\delta^{-1}\gamma\Delta^{-1}=\alpha\Delta^{-1},\quad 1+\Delta^{-1}\beta\delta^{-1}\gamma=\Delta^{-1}\alpha,$$
which will be used without further reference.

Let $$\varphi'=\left(\begin{array}{cc}\alpha' &\beta' \\ \gamma' & \delta'\end{array}\right)$$ be the matrix defined in the thesis. We only prove that $\varphi\varphi'=1$, since $\varphi'\varphi=1$ can be obtained in a not very different way.
\begin{align*}
\alpha\alpha'+\beta\gamma'&=\alpha\Delta^{-1}-\beta\delta^{-1}\gamma\Delta^{-1}=(\alpha-\beta\delta^{-1}\gamma)\Delta^{-1}=\Delta\Delta^{-1}=1.\\[1em]
\alpha\beta'+\beta\delta'&=\alpha(-\Delta^{-1}\beta\delta^{-1}) +\beta(1+\delta^{-1}\gamma\Delta^{-1}\beta)\delta^{-1}=\\
&=-(1+\beta\delta^{-1}\gamma\Delta^{-1})\beta\delta^{-1}+\beta(1+\delta^{-1}\gamma\Delta^{-1}\beta)\delta^{-1}=0.\\[1em]
\gamma\alpha'+\delta\gamma'&=\gamma\Delta^{-1}-\delta(\delta^{-1}\gamma\Delta^{-1})=\gamma\Delta^{-1}-\gamma\Delta^{-1}=0.\\[1em]
\gamma\beta'+\delta\delta'&=\gamma(-\Delta^{-1}\beta\delta^{-1})+\delta(1+\delta^{-1}\gamma\Delta^{-1}\beta)\delta^{-1}=\\
&=-\gamma\Delta^{-1}\beta\delta^{-1}+1+\gamma\Delta^{-1}\beta\delta^{-1}=1
\end{align*}
Then $\varphi'$ is the inverse of $\varphi$.

Finally, $\alpha'$ is invertible and hence $\det_K\varphi'$ can be defined. Its value can be computed directly or just derived applying $\varphi=(\varphi^{-1})^{-1}$. Moreover, the inverse of a homomorphism is always a homomorphism, so that $\varphi'$ belongs to $\mathcal{M}_{H,K}$. Then $\Delta^{-1}$ and, hence, $\Delta$ are endomorphisms of $H$.
\end{proof}

From this and from Lemma \ref{norm} the following information can be derived.

\begin{cor}
Let $H$ and $K$ be groups and let $\varphi$ be an automorphism of $H\times K$ such that either $\delta$ and $\det_H\varphi$ are invertible or $\alpha$ and $\det_K\varphi$ are invertible. Then $\im\beta\leq Z(H)$ and $\im\gamma\leq Z(K)$. 
\end{cor}

This leads us to define the following particular set of endomorphisms (see \cite{BidCurMcC} or \cite{John}), which will be seen to have many useful properties. Let $H$ and $K$ be groups and take into account the set $$\mathcal{A}_{H,K}=\left\{\left(\begin{array}{cc}\alpha &\beta \\ \gamma & \delta\end{array}\right):\begin{array}{ll} \alpha\in \aut H, &\beta\in \Hom(K,Z(H)) \\
\gamma\in \Hom(H,Z(K)), &\delta\in \aut K
\end{array}
\right\}.$$

As can be found in Example \ref{Ex1} and in Example \ref{ExTotIncPer}, $\mathcal{A}_{H,K}$ is in general not a group, even in the event that $H$ and $K$ are abelian groups (torsion-free and torsion, respectively) sharing no isomorphic non-trivial direct summand. However, as we will see, in some relevant cases $\mathcal{A}_{H,K}$ is isomorphic with the whole $\aut(H\times K)$. To inspect further these possibilities, the concept of determinant will be of use.

\smallskip 

Before going on, we describe $\mathcal{A}_{H,K}$ as a product of subgroups, which can be employed to compute its cardinality or its structure in some particular cases. Although the proof can already be retrieved from \cite{BidCurMcC}, it is easy and we give it here in a more general setting.

\begin{prop}\label{AStruc}
Let $H$ and $K$ be groups. Let $U$ and $L$ be the subgroups of upper and lower unitriangular matrices of $\mathcal{A}_{H,K}$, respectively, and let $D_1$ and $D_2$ be the diagonal matrices of $\mathcal{A}_{H,K}$ acting trivially on $K$ and $H$, respectively. If $\mathcal{A}_{H,K}$ is a subgroup of $\aut(H\times K)$, then $\mathcal{A}_{H,K}=D_1ULD_2$, $D_1\cap D_2=1$ and both $U$ and $L$ are normalized by $D_1\times D_2$.
\end{prop}
\begin{proof}
Let $\beta\in\Hom(K,Z(H))$ and $\gamma\in\Hom(H,Z(K))$. Since $\mathcal{A}_{H,K}$ is a group, then $\left(\begin{array}{cc}1 &-\beta \\ 0 & 1\end{array}\right)\left(\begin{array}{cc}1 &0 \\ \gamma & 1\end{array}\right)=\left(\begin{array}{cc}1-\beta\gamma &* \\ * & *\end{array}\right)$ is still an element of $\mathcal{A}_{H,K}$ and hence $1-\beta\gamma$ is an automorphism of $H\times K$. Then we may write $$\left(\begin{array}{cc}1 &\beta \\ \gamma & 1\end{array}\right)=\left(\begin{array}{cc}1-\beta\gamma &0 \\ 0 & 1\end{array}\right)\left(\begin{array}{cc}1 &(1-\beta\gamma)^{-1}\beta \\ 0 & 1\end{array}\right)\left(\begin{array}{cc}1 &0 \\ \gamma & 1\end{array}\right)$$ which is an element of $D_1UL$. On the other hand, if $\alpha\in\aut H$ and $\delta\in\aut K$, we may write $$\left(\begin{array}{cc}\alpha &\beta \\ \gamma & \delta\end{array}\right)=\left(\begin{array}{cc}\alpha &0 \\ 0 & 1\end{array}\right)\left(\begin{array}{cc}1 &\alpha^{-1}\beta\delta^{-1} \\ \gamma & 1\end{array}\right)\left(\begin{array}{cc}1 &0 \\ 0 & \delta\end{array}\right)$$ which is an element of $D_1(D_1UL)D_2=D_1ULD_2$. Finally, it is straightforward to check that that $D_1$ and $D_2$ have trivial intersection and that $D_1\times D_2$ normalizes both $U$ and $L$, so the proof is complete.
\end{proof}

\smallskip

We now come back to the determinant, starting with an easy and curious lemma.

\begin{lemma}\label{detiff}
Let $H$ and $K$ be groups and let $\varphi$ be an element of $\mathcal{A}_{H,K}$. Then the following are equivalent
\begin{itemize}
\item[(1)] $\det_H\varphi$ is invertible, and
\item[(2)] $\det_K\varphi$ is invertible.
\end{itemize}
Moreover, if these conditions hold, we have that $$(\det\!{}_H\varphi)^{-1}=\alpha^{-1}\beta(\det\!{}_K\varphi)^{-1}\gamma\alpha^{-1} + \alpha^{-1}$$ and $$(\det\!{}_K\varphi)^{-1}=\delta^{-1}\gamma(\det\!{}_H\varphi)^{-1}\beta\delta^{-1} + \delta^{-1}.$$
\end{lemma}
\begin{proof}
By symmetry, we need only prove that (1) implies (2), and that (1) and (2) imply the first equality of the statement. Assume that $\delta-\gamma\alpha^{-1}\beta$ is invertible. Let $h$ be an element of $H$ such that $\alpha-\beta\delta^{-1}\gamma(h)=1$, namely, in other words, that $\alpha(h)=\beta\delta^{-1}\gamma(h)$. Since $\delta-\gamma\alpha^{-1}\beta$ is injective, form $$\delta-\gamma\alpha^{-1}\beta(\delta^{-1}\gamma(h))=\gamma-\gamma\alpha^{-1}\beta\delta^{-1}\gamma(h)=\gamma-\gamma(h)=1$$ it follows that $\delta^{-1}\gamma(h)=1$. From this we get $\alpha(h)=\beta(1)=1$ and hence $h=1$. Therefore, $\alpha-\beta\delta^{-1}\gamma$ is injective.

Let $h$ be an element of $H$ and let $k$ be an element of $K$ such that $\delta-\gamma\alpha^{-1}\beta(k)=\gamma\alpha^{-1}(h)$, namely that $-\gamma\alpha^{-1}\beta(k)=\delta(k^{-1})\gamma\alpha^{-1}(h)$. Then 
\begin{align*}
\alpha-\beta\delta^{-1}\gamma(\alpha^{-1}\beta(k)\alpha^{-1}(h))&=\beta(k)h(\beta\delta^{-1}\gamma\alpha^{-1}\beta(k^{-1}))(\beta\delta^{-1}\gamma\alpha^{-1}(h^{-1}))=\\
&=\beta(k)h\beta(k^{-1})(\beta\delta^{-1}\gamma\alpha^{-1}(h))(\beta\delta^{-1}\gamma\alpha^{-1}(h^{-1}))=h
\end{align*}
This shows that $\alpha-\beta\delta^{-1}\gamma$ is also surjective. In particular, from the proof of surjectivity one can get that $$(\alpha-\beta\delta^{-1}\gamma)^{-1}=\alpha^{-1}\beta(\delta-\gamma\alpha^{-1}\beta)^{-1}\gamma\alpha^{-1} + \alpha^{-1}$$ and we are done.
\end{proof}

On the other hand, notice that if $H$ is isomorphic with $K$ and both $H$ and $K$ are non-trivial, $\aut (H\times K)$ always contains an automorphism such that neither $\alpha$ nor $\delta$ are invertible.

Lemma \ref{detiff} suggests the following definition: let $H$ and $K$ be groups and let $\varphi$ be an element of $\mathcal{A}_{H,K}$. We let $\alpha-\beta\delta^{-1}\gamma$ be the \textit{determinant} of $\varphi$ and denote it by $\det\varphi$.

\begin{lemma}\label{inv}
Let $H$ and $K$ be groups, let $\varphi$ be an element of $\mathcal{A}_{H,K}$ and let $\Delta=\det\varphi$. If $\Delta$ is invertible, then $$\left(\begin{array}{cc}\Delta^{-1} &-\Delta^{-1}\beta\delta^{-1} \\ -\delta^{-1}\gamma\Delta^{-1} & (1+\delta^{-1}\gamma\Delta^{-1}\beta)\delta^{-1}\end{array}\right)$$ is the inverse of $\varphi$. Moreover, $\varphi^{-1}$ belongs to $\mathcal{A}_{H,K}$ and $\det\varphi^{-1}=\alpha^{-1}$.
\end{lemma}
\begin{proof}
The formula for $\varphi^{-1}$ can be immediately derived from Proposition \ref{invgen}. Since by Lemma \ref{detiff}, also $\det_K\varphi$ is invertible, then we can compare the two forms of $\varphi^{-1}$ in Proposition \ref{invgen} to get that $(1+\delta^{-1}\gamma\Delta^{-1}\beta)\delta^{-1}=\det_K^{-1}$, so that $\varphi^{-1}$ is an element of $\mathcal{A}_{H,K}$.

\begin{comment}
PECCATO, era un conto carino.
Let $\Delta_K=\delta-\gamma\alpha^{-1}\beta$. From $\Delta_K^{-1}(\delta-\gamma\alpha^{-1}\beta)=1$, one get the following relations

$$1+\gamma\alpha^{-1}\beta\Delta_K^{-1}=\delta\Delta_K^{-1}\quad\text{ and }\quad 1+\Delta_K^{-1}\gamma\alpha^{-1}\beta=\Delta_K^{-1}\delta.$$

To show that $\varphi^{-1}$ belongs to $\mathcal{A}_{H,K}$, we only need to prove that $1+\delta^{-1}\gamma\Delta^{-1}\beta$ is bijective. From Lemma \ref{detiff} and from the relations above we obtain, we have 
\begin{align*}
1+\delta^{-1}\gamma\Delta^{-1}\beta&=1+\delta^{-1}\gamma(\alpha^{-1}\beta\Delta_K^{-1}\gamma\alpha^{-1} + \alpha^{-1})\beta=\\
&=1+\delta^{-1}\gamma\alpha^{-1}\beta\Delta_K^{-1}\gamma\alpha^{-1}\beta + \delta^{-1}\gamma\alpha^{-1}\beta=\\
&=1+\delta^{-1}(\delta\Delta_K^{-1}-1)\gamma\alpha^{-1}\beta + \delta^{-1}\gamma\alpha^{-1}\beta=\\
&=1+\Delta_K^{-1}\gamma\alpha^{-1}\beta-\delta^{-1}\gamma\alpha^{-1}\beta + \delta^{-1}\gamma\alpha^{-1}\beta=\Delta_K^{-1}\delta
\end{align*}

\end{comment}

Finally, $\det\varphi^{-1}$ can be computed directly or just derived applying $\varphi=(\varphi^{-1})^{-1}$.
\end{proof}

From the fact that $1+\delta^{-1}\gamma\Delta^{-1}\beta=\Delta_K^{-1}\delta$ we obtain the following somewhat more pleasant formula for inverse automorphisms in $\mathcal{A}_{H,K}$.

\begin{prop}
Let $H$ and $K$ be groups, let $\varphi$ be an element of $\mathcal{A}_{H,K}$. If $\det\varphi$ is invertible, then $$\left(\begin{array}{cc}\Delta_H^{-1} &-\alpha^{-1}\beta\Delta_K^{-1} \\ -\delta^{-1}\gamma\Delta_H^{-1} & \Delta_K^{-1}\end{array}\right)$$ is the inverse of $\varphi$.
\end{prop}

We are now ready to state one of our main results about the determinant.

\begin{theorem}\label{propinv}
Let $H$ and $K$ be groups and let $\varphi$ be an endomorphism of $H\times K$ such that $\delta$ ($\alpha$, respectively) is invertible. Then $\varphi$ is invertible if and only if $\det_H\varphi$ ($\det_K\varphi$, respectively) is invertible.
\end{theorem}
\begin{proof}
As the two cases for $\alpha$ and $\delta$ are completely symmetrical, we prove only the one where $\delta$ is invertible. Assume first that $\varphi$ is invertible and let $h$ be an element of $H$ such that $\alpha-\beta\delta^{-1}\gamma(h)=1$. Then it is immediate to check that $$\left(\begin{array}{cc}\alpha &\beta \\ \gamma & \delta\end{array}\right)\left(\begin{array}{c}h \\ -\delta^{-1}\gamma(h)\end{array}\right)=\left(\begin{array}{c}1 \\ 1\end{array}\right)$$ and hence $h=1$. Let now $h$ be any element of $H$ and let $x$ and $y$ be elements of $H$ and $K$, respectively, such that $$\left(\begin{array}{cc}\alpha &\beta \\ \gamma & \delta\end{array}\right)\left(\begin{array}{c}x \\ y\end{array}\right)=\left(\begin{array}{c}h \\ 1\end{array}\right)$$ From this one gets that $\alpha-\beta\delta^{-1}\gamma(x)=h$ and hence $\alpha-\beta\delta^{-1}\gamma$ is invertible.

The converse, follows from Proposition \ref{invgen}.
\end{proof}

\begin{cor}
Let $H$ and $K$ be groups and let $\varphi$ be an element of $\mathcal{A}_{H,K}$. Then $\varphi$ is invertible if and only if $\det\varphi$ is invertible.
\end{cor}

We conclude this section giving a necessary and sufficient condition for $\mathcal{A}_{H,K}$ to be a subgroup of $\aut(H\times K)$.

\begin{prop}\label{propsub}
Let $H$ and $K$ be groups. Then $\mathcal{A}_{H,K}$ is a subgroup of $\aut(H\times K)$ if and only if for every $\lambda\in\aut H$, $\mu\in \Hom(H,Z(K))$, $\nu\in\aut K$ and $\xi\in \Hom(K,Z(H))$, $\lambda+\xi\mu$ and $\nu+\mu\xi$ are bijective.
\end{prop}
\begin{proof}
Assume first that $\mathcal{A}_{H,K}$ is a subgroup of $\aut(H\times K)$ and let $\lambda\in\aut H$, $\mu\in \Hom(H,Z(K))$, $\nu\in\aut K$ and $\xi\in \Hom(K,Z(H))$. Then $$\left(\begin{array}{cc}\lambda &\xi \\ 0 & 1\end{array}\right)\left(\begin{array}{cc}1 &0 \\ \mu & 1\end{array}\right)=\left(\begin{array}{cc}\lambda+\xi\mu &* \\ * & *\end{array}\right)$$ is still an element of $\mathcal{A}_{H,K}$ and hence $\lambda+\xi\mu$ is an automorphism of $H\times K$. Since the same holds for $\nu+\mu\xi$, the necessary condition is proved.

Conversely, let $\varphi=\left(\begin{array}{cc}\alpha &\beta \\ \gamma & \delta\end{array}\right)$ be an element of $\mathcal{A}_{H,K}$. Since the hypothesis yields that $\det\varphi=\alpha-\beta\delta^{-1}\gamma$ is invertible, $\varphi$ is invertible by Lemma \ref{inv} and $\varphi^{-1}$ belongs to $\mathcal{A}_{H,K}$, which hence contains the inverses of all its elements. Let now $\chi=\left(\begin{array}{cc}\varepsilon &\zeta \\ \eta & \theta\end{array}\right)$ be another element of $\mathcal{A}_{H,K}$ and take into account the product $\varphi\chi$. From the hypothesis we have that $\alpha\varepsilon+\beta\eta$ and $\gamma\zeta+\delta\theta$ are bijective, while it is clear that $\alpha\zeta+\beta\theta\in\Hom(K,Z(H))$ and that $\gamma\varepsilon+\delta\eta\in\Hom(H,Z(K))$. So $\varphi\chi$ is an element of $\mathcal{A}_{H,K}$ and $\mathcal{A}_{H,K}$ is a subgroup of $\aut(H\times K)$.
\end{proof}

\smallskip

\subsection{The general case: a sketch}\label{SubSketch}

Here we give the definition of determinants for automorphisms of a generic direct product of finitely many groups and sketch some properties. In the following we will denote with $I_k=\{1,\ldots,k\}$ be the set of the first $k$ positive integers for any positive integer $k$.

Let $H_1,\ldots, H_n$ be monoids. In analogy with $\mathcal{M}_{H,K}$ and $\mathcal{A}_{H,K}$ defined above, we define the following sets (see, for instance, \cite{Bid} or \cite{John}) $$\mathcal{M}_{H_1,\ldots,H_n}=\left\{\!\left(\begin{array}{ccc}\varphi_{1,1} &\cdots & \varphi_{1,n} \\ \vdots & \ddots & \vdots \\ \varphi_{n,1} & \cdots & \varphi_{n,n} \end{array}\right):\!\begin{array}{ll} \varphi_{i,j}\in \Hom(H_j,H_i), &\!\forall i,j\in I_n \\
\,\![\im\varphi_{i,j},\im\varphi_{i,k}]=1 & \!\forall i,j,k\in I_n\text{ and } j\neq k
\end{array}
\!\!\right\}$$
and
$$\mathcal{A}_{H_1,\ldots,H_n}=\left\{\!\left(\begin{array}{ccc}\varphi_{1,1} &\cdots & \varphi_{1,n} \\ \vdots & \ddots & \vdots \\ \varphi_{n,1} & \cdots & \varphi_{n,n} \end{array}\right):\!\begin{array}{ll} \varphi_{i,i}\in \aut(H_i), &\forall i\in I_n \\
\varphi_{i,j}\in \Hom(H_j,Z(H_i)) & \forall i,j\in I_n\text{ and } i\neq j
\end{array}
\!\!\right\}$$

Just as for Proposition \ref{end}, the following can be proved with some routine calculations.

\begin{theorem}\label{endGen}
Let $H_1,\ldots, H_n$ be monoids. Then $\End(H_1\times\cdots\times H_n)\simeq \mathcal{M}_{H_1,\ldots,H_n}$.
\end{theorem}

\smallskip

In general, the concept of determinant can be defined starting from the case $n=2$. Let $m$ and $n$ be positive integers such that $m\leq n$. Here an $m\times m$ matrix with values in a set $S$ will be just a function from $I\times I$ to $S$, where $I$ is a subset of order $m$ of $I_n$. This way we may have matrices "missing some rows and columns", which will make notation less cumbersome for our scopes. Let $H_1,\ldots, H_n$ be groups and let $$\varphi= \left(\begin{array}{ccc}\varphi_{1,1} &\cdots & \varphi_{1,n} \\ \vdots & \ddots & \vdots \\ \varphi_{n,1} & \cdots & \varphi_{n,n} \end{array}\right)$$ be an element of $\mathcal{M}_{H_1,\ldots,H_n}$. Let $k$ be a non-negative integer $k\leq n$ and let $F_k$ be the set of all injective functions from $I_k$ to $I_n$. Notice that $F_0$ has only one element, i.e. the empty function on an empty domain. With an abuse of notation, $F_0$ will be denoted by $\{\emptyset\}$.
%Perché in realtà la funzione vuota è $(\emptyset,\emptyset)$
For each $i,j\in I_n$ define $(\det^{\emptyset}\!\varphi)_{i,j}=\varphi_{i,j}$ or, in other words, $\det^{\emptyset}\!\varphi=\varphi$. In particular, $\det^{\emptyset}\!\varphi$ is an $n\times n$ matrix. Let $k$ be an element of $\{0,\ldots,n-2\}$ and assume that, if $g$ is an element of $F_k$ and $\det^g\varphi$ can be defined, then the latter is an $(n-k)\times(n-k)$ matrix whose domain is $(I_n\setminus \im g)\times(I_n\setminus \im g)$ and such that for every $i,j\in I_n\setminus \im g$, $(\det^g\varphi)_{i,j}\in Hom(H_j,H_i)$. If now $m$ is an element of $I_n\setminus \im g$, if $(\det^g\varphi)_{m,m}$ is invertible and if $f$ is the extension of $g$ to $I_{k+1}$ sending $k+1$ to $m$, then we set  
$$(\det\nolimits^f\!\varphi)_{i,j}=(\det\nolimits^g\varphi)_{i,j}-(\det\nolimits^g\!\varphi)_{i,m}(\det\nolimits^g\!\varphi)_{m,m}^{-1}(\det\nolimits^g\!\varphi)_{m,j}$$ for every $i,j\in I_n\setminus \im f$. If there is an element $f$ of $F_{n-1}$ such that $\det^f\!\varphi$ can be defined, then $\det^f\!\varphi$ will be called the $f$-\textit{determinant} of $\varphi$. It is easy to check that $\det^f\!\varphi$ is an endomorphism of $H_s$, where $s$ is the only integer in $I_n\setminus \im f$. In particular, if $n=2$ and $f$ is the immersion of $I_1$ into $I_2$, then the $f$-determinant of $\varphi$ is the $H_2$-determinant defined in the previous subsection; on the other hand, if $f: 1\in I_1\mapsto 2\in I_2$, then $\det^f\!\varphi=\det_{H_1}\varphi$.

From this definition, applying some straightforward calculations and proceeding by induction on $n$, it is possible to prove the following analogue to Theorem \ref{propinv}.

\begin{theorem}\label{propinvGen}
Let  $H_1,\ldots, H_n$ be groups and let $\varphi$ be an endomorphism of $H_1\times\ldots\times H_n$. Then $\varphi$ is invertible if and only if there exist an element of $F_{n-1}$ such that $\det^f\!\varphi$ can be defined and is invertible.
\end{theorem}

As one would expect, notation can be lightened for $\mathcal{A}_{H_1,\ldots, H_n}$. Indeed, an analogue to Lemma \ref{detiff} can be proved, so that one can define the \textit{determinant} of an element $\varphi$ of $\mathcal{A}_{H_1,\ldots, H_n}$ as $\det\varphi=\det^f\!\varphi$, where $f:x\in I_{n-1}\mapsto n-x+1\in I_{n}$. Clearly, $\det\varphi$ is an endomorphism of $H_1$. Then we get the following

\begin{cor}\label{CorDetGen}
Let $H_1,\ldots, H_n$ be groups and let $\varphi$ be an element of $\mathcal{A}_{H_1,\ldots, H_n}$. Then $\varphi$ is invertible if and only $\det\varphi$ is invertible.
%Qui ho tolto "can be defined and", che era scritto non capisco perché
\end{cor}

Thinking of the way the determinants are defined, also Proposition \ref{propsub} can be extended to give a useful condition under which $\mathcal{A}_{H_1,\ldots, H_n}$ is a subgroup of $\mathcal{M}_{H_1,\ldots, H_n}$. 

\begin{prop}
Let $H_1,\ldots, H_n$ be groups. Then $\mathcal{A}_{H_1,\ldots, H_n}$ is a subgroup of the automorphism group of $H_1\times\ldots\times H_n$ if and only if for every $i,j\in I_n$ such that $i\neq j$ and for every $\lambda\in\aut H_i$, $\mu\in \Hom(H_i,Z(H_j))$, $\nu\in\aut H_j$ and $\xi\in \Hom(H_j,Z(H_i))$, $\lambda+\xi\mu$ and $\nu+\mu\xi$ are bijective.
\end{prop}

The usefulness of this condition is to be seen in connection with the concepts of incompatibility which are going to be defined in next section. In particular, an extended concept of incompatibility can be defined for direct products of any finite number of groups, but this goes beyond the scope of the present paper.

\section{Group pairs}\label{SecPairs}

In this section we are going to find some conditions on the groups $H$ and $K$, which will allow us to give good descriptions of $\aut(H\times K)$ and of $\mathcal{A}_{H,K}$. To do so, we define some particular pairs of groups and give some basic implementation of the concept of determinant.

\subsection{Incompatible pairs}

Let $H$ and $K$ be groups. Starting from Lemma \ref{norm}, a way to a better understanding of $\aut(H\times K)$ can be attained when the sets of normal subgroups of $H$ and $K$ are well-behaved, in a sense that will be made precise soon. In this light, the following easy lemma will be useful.

\begin{lemma}\label{min} Let $G$ be a group satisfying the minimal condition on normal subgroups and let $\varphi$ be a normal endomorphism of $G$. Then $\varphi$ is injective if and only if is bijective.
\end{lemma}
\begin{proof}
Assume for a contradiction that $\varphi$ is not surjective and let $G_1=\im\varphi$. Then $G_1$ is a proper subgroup of $G$ and it is also normal in $G$ because $\varphi$ is normal. Since $\varphi$ is injective, $\varphi(G_1)$ must be a proper subgroup of $G_1$ which is still normal in $G$ because of the normality of $\varphi$. Iterating this process leads to a contradiction. 
\end{proof}

%Non ha senso metterlo qui. Boh. We remark here that there exist simple groups which have proper subgroups which are isomorphic with the whole group.
%Credo PSL(2,\mathbb{R})
%Hence, a converse for the above lemma does not hold.
This lemma shows that prescribing the minimal condition on normal subgroups gives us surjectivity, which is a feature we want to pursue in view of Lemma \ref{norm} (1). Moreover, again referring to Lemma \ref{norm} (1), if $\alpha$ is not surjective it is easy to construct automorphisms such that $\im\beta$ is not central: take for instance $H=K\simeq S_3$, i.e. the symmetric group on $3$ elements and the automorphism of $H\times K$ swapping $H$ and $K$. This kind of swapping displays a typical behaviour of the two groups $H$ and $K$ which we want to avoid,
%?
and this can be achieved by giving the following definition.

For short, we will say that two groups $H$ and $K$ are \textit{incompatible} or, equivalently, that the pair $(H,K)$ is \textit{incompatible}, if for every $\sigma\in\Hom(H,K)$, $\tau\in\Hom(K,H)$, $h\in H$ and $k\in K$, such that $\sigma\tau$ and $\tau\sigma$ are normal, one of the following equivalent conditions hold:
\begin{itemize}
\item[(1)] If $\sigma\tau(k)=k$, then $k=1$;
\item[(2)] If $\tau\sigma(h)=h$, then $h=1$.
\end{itemize}
Indeed, if $\sigma\tau$ fixes a non-trivial element $k$, this means that $\tau(k)$, which cannot be trivial, is fixed by $\tau\sigma$. And vice versa. Two groups which are not incompatible will be said to be \textit{compatible}. In full, $H$ and $K$ are compatible if there exist $\sigma\in\Hom(H,K)$, $\tau\in\Hom(K,H)$, $h\in H$ and $k\in K$ such that $\sigma\tau$ and $\tau\sigma$ are normal and such that either $h$ is non-trivial and $\tau\sigma(h)=h$ or $k$ is non-trivial and $\sigma\tau(k)=k$. Moreover, $H$ and $K$ (or the pair $(H,K)$) will be said to be \textit{centrally incompatible} if $H$ and $K$ satisfy the same conditions for incompatibility, but with $\Hom(H,Z(K))$ and $\Hom(K,Z(H))$ in place of $\Hom(H,K)$ and $\Hom(K,H)$, respectively. Two groups which are not centrally incompatible will be said to be \textit{centrally compatible}.

We are now showing how incompatibility plays a role in determining the automorphism group of a direct product of two groups.

\begin{theorem}\label{incompatible}
Let $H$ and $K$ be incompatible groups. If $H$ or $K$ satisfies the minimal condition on normal subgroups, then $\mathcal{A}_{H,K}$ is a subgroup of $\aut (H\times K)$. If both $H$ and $K$ satisfy the minimal condition on normal subgroups, then $\aut (H\times K)\simeq\mathcal{A}_{H,K}$.
\end{theorem}
\begin{proof}
Assume without loss of generality that $H$ satisfies the minimal condition on normal subgroups and let $\lambda\in\aut H$, $\mu\in \Hom(H,Z(K))$, $\nu\in\aut K$ and $\xi\in \Hom(K,Z(H))$. If $h$ is an element of $H$ such that $\lambda-\xi\mu(h)=1$, so that $\lambda^{-1}\xi\mu(h)=h$. However, $\lambda^{-1}\xi\mu$ is a normal endomorphism of $H$, because its image is contained in $Z(H)$, and hence $h=1$ by the incompatibility hypothesis. Then $\lambda-\mu\xi$ is injective. Applying exactly the same argument, we have also that $\nu-\mu\xi$ is injective. Moreover, by the hypothesis on $H$ and by Lemma \ref{min}, $\lambda-\xi\mu$ is bijective. Let now $k$ be an element of $K$, let $m$ be the order of $\xi(k)$ and put $\theta=\nu^{-1}\mu\xi$, which is a normal endomorphism of $K$. As $H$ satisfies the minimal condition on normal subgroups, the subgroup $M$ generated by the elements of $Z(H)$ of order at most $m$ is finite. Since $M$ is clearly also fully invariant, by incompatibility there is a non-negative integer $n$ such that $(\mu\xi\nu^{-1})^n(M)=1$, otherwise there would be a non-trivial element $z$ of $M$ and a positive integer $l$ such that $(\mu\xi(\theta^{l-1}\nu^{-1}))(z)=(\mu\xi\nu^{-1})^l(z)=z$, contradicting the fact that $H$ and $K$ are incompatible. Then, it is not difficult to see that $s=\nu^{-1}+\theta\nu^{-1}+\cdots+\theta^{n-1}\nu^{-1}(k)$ is a solution to the equation $\nu-\mu\xi(x)=k$, so that $\nu-\mu\xi$ is also surjective. Now, from the generality of $\lambda,\mu,\nu$ and $\xi$ and from Proposition \ref{propsub} (the sign does not count, as $\im\mu$ and $\im\xi$ are subgroups of $Z(K)$ and of $Z(H)$, respectively), we have that $\mathcal{A}_{H,K}$ is a subgroup of $\aut (H\times K)$.

\smallskip

Suppose now that both $H$ and $K$ satisfy the minimal condition on normal subgroups. Let $\left(\begin{array}{cc}\alpha &\beta \\ \gamma & \delta\end{array}\right)$ be an automorphism of $H\times K$, whose inverse will be as usual denoted with $\left(\begin{array}{cc}\alpha' &\beta' \\ \gamma' & \delta'\end{array}\right)$. First, we prove that $\alpha$ and $\delta$ are injective. Let hence $h$ be an element of $H$ such that $\alpha(h)=1$. By Lemma \ref{cases} (2.1), we may write $h=\alpha'\alpha(h)\beta'\gamma(h)=\beta'\gamma(h)$. Using Lemma \ref{norm} (2), we see that $\beta'\gamma$ is a normal endomorphism of $H$ and thus we can use the incompatibility of $H$ and $K$ to infer that $h=1$. Making use of Lemma \ref{cases} (2.4) and of Lemma \ref{norm} (3), we have that also $\delta$ is injective. Now by hypothesis, Lemma \ref{norm} and Lemma \ref{min}, $\alpha$ and $\delta$ are also surjective, so that from the relations in $\mathcal{M}_{H,K}$ we obtain that $\im\beta\subseteq Z(H)$ and $\im\gamma\subseteq Z(K)$. We have shown that $\aut (H\times K)\subseteq\mathcal{A}_{H,K}$ and hence the thesis follows.
%Let now... and let $(h,k)$ be an element in $\ke\varphi$. This amounts to say that $\alpha(h)\beta(k)=1$ and $\gamma(h)\delta(k)=1$. From these relations we get that $$h=(\alpha^{-1}\beta)(\delta^{-1}\gamma)(h)\quad\text{and}\quad k=(\delta^{-1}\gamma)(\alpha^{-1}\beta)(k)$$ Now, $\beta$ and $\gamma$ are normal endomorphisms, since their images are central, while $\alpha^{-1}$ and $\delta^{-1}$ are also obviously normal. Since $H$ and $K$ are incompatible, this yields that $(h,k)=(1,1)$. Hence every element of $\mathcal{A}_{H,K}$ is an injective endomorphism of $H\times K$. However, the fact that both $H$ and $K$ satisfy the minimal condition on normal subgroups implies that the same holds for $H\times K$ (see, for instance, \cite[3.1.7]{RobCour}) and hence by Lemma \ref{min} we have the thesis.
\end{proof}

A different version of this theorem is possible for centrally incompatible groups.

\begin{prop}\label{PropCInc} Let $H$ and $K$ be groups. If $H$ or $K$ satisfies the minimal condition on normal subgroups, then $H$ and $K$ are centrally incompatible if and only if $\mathcal{A}_{H,K}$ is a subgroup of $\aut (H\times K)$.
\end{prop}
\begin{proof}
We can assume without loss of generality that $H$ satisfies the minimal condition on normal subgroups. The necessity condition runs exactly the same way as the first part of the proof of Theorem \ref{incompatible}. Conversely, assume that $H$ and $K$ are centrally compatible, viz. that there exist $\sigma\in\Hom(H,Z(K))$, $\tau\in\Hom(K,Z(H))$ and, without loss of generality, $h\in H\setminus\{1\}$ such that $\tau\sigma(h)=h$. If we let $k=\sigma(h)$, we get in particular that $h$ is central in $H$ and $k$ in $K$. If we now define $\beta=-\tau$ and $\gamma=-\sigma$ and $\varphi=\left(\begin{array}{cc}1 &\beta \\ \gamma & 1\end{array}\right)$ we immediately see that the non-trivial element $(h,k)$ of $H\times K$ belongs to $\ke\varphi$, so that $\varphi$ is an element of $\mathcal{A}_{H,K}$ which is not an automorphism of $H\times K$.
\end{proof}

The hypothesis on the minimal condition on normal subgroups in Theorem \ref{incompatible} cannot be removed, as Examples \ref{Ex3} and \ref{Ex4} show. On the other hand, very easy examples, such as the direct product of any non-trivial group by itself, show that also the incompatibility hypothesis is essential. One of our next targets will be that of  make incompatibility simpler, possibly at the expenses of adding some other hypotheses. Next lemma provides a framework to work within.

\begin{lemma}\label{LemComm}
Let $H$ and $K$ be groups which have some common non-trivial direct factor. Then $H$ and $K$ are compatible. Moreover, $\aut(H\times K)$ is not contained in $\mathcal{A}_{H,K}$.
\end{lemma}
\begin{proof}
Let $X$ and $Y$ be non-trivial direct factors of $H$ and $K$, respectively, such that $\varphi$ is an isomorphism from $X$ to $Y$. Write $H=X\times M$ and $K=Y\times N$, let $\pi_X$, $\pi_M$, $\pi_Y$, $\pi_N$ be the natural projections of $H$ onto $X$ and $M$ and of $K$ onto $Y$ and $N$, respectively, and let $\iota_X$, $\iota_M$, $\iota_Y$, $\iota_N$ be the immersion maps of $X$ and $M$ in $H$ and of $Y$ and $N$ in $K$, respectively. Then it is clear that $\iota_Y\varphi\pi_X$ and $\iota_X\varphi^{-1}\pi_Y$ provide compatibility for $H$ and $K$ for any non-trivial element of $X$ or $Y$. Moreover, the endomorphism $$\left(\begin{array}{cc}\iota_M\pi_M &\iota_Y\varphi\pi_X \\ \iota_X\varphi^{-1}\pi_Y & \iota_N\pi_N\end{array}\right)$$ is easily seen to be an automorphism of $H\times K$ which does not belong to $\mathcal{A}_{H,K}$, because, for instance, $\iota_M\pi_X$ is not bijective.
\end{proof}

This suggests to investigate situations in which the implication can be inverted. In fact, we will show in subsection \ref{subsmaxminn} that incompatibility is equivalent to having no isomorphic non-trivial direct factors, at least in the universe of groups satisfy\-ing the maximal and the minimal condition on normal subgroups.

\bigskip

\subsection{Totally incompatible pairs}

Let $H$ and $K$ be groups. Previous results exhibited some cases in which $\mathcal{A}_{H,K}$ is a subgroup of $\aut(H\times K)$ and the fact was strongly related to the normal structure of at least one between $H$ and $K$. In fact, the property of being incompatible is quite general and without some restriction on the normal structure of either $H$ or $K$, not much more can be said (see Example \ref{Ex1}). In order to investigate the possibility for $\mathcal{A}_{H,K}$ to be a subgroup of $\aut(H\times K)$ in a more general setting, we introduce some definitions, which will put to work Proposition \ref{propsub}.

We will say that two groups $H$ and $K$ are \textit{totally incompatible} or, equivalently, that the pair $(H,K)$ is \textit{totally incompatible}, if for every $\sigma\in\Hom(H,K)$, $\tau\in\Hom(K,H)$, $h\in H$ and $k\in K$, when $\sigma\tau$ and $\tau\sigma$ are normal there exists a positive integer $n$ (depending on $\sigma$, $\tau$, $h$ and $k$) such that one of the following equivalent conditions hold:
\begin{itemize}
\item[(1)] $(\sigma\tau)^n(k)=1$;
\item[(2)] $(\tau\sigma)^n(h)=1$.
\end{itemize}
%Se non facciamo dipendere n da h o k, allora possiamo dire che i due sono 0
Indeed, it is immediate to see that if $(\sigma\tau)^n(k)=1$ for some $n\geq1$ and $k\in K$, then $(\tau\sigma)^{n+1}(k)=1$. In particular, a totally incompatible pair is also incompatible. Thanks to this definition, we may dismiss some hypotheses on the normal structure of the groups involved.

\begin{prop}\label{htotincompatible}
Let $H$ and $K$ be totally incompatible groups. Then $\aut (H\times K)\simeq\mathcal{A}_{H,K}$.
\end{prop}
\begin{proof}
We begin proving that $\mathcal{A}_{H,K}\subseteq\aut (H\times K)$.  Let $\lambda$ be an automorphism of $H$ and let $\mu$ be an endomorphism of $H$ which can be decomposed as the composition of an homomorphism in $\Hom(H,Z(K))$ and one in $\Hom(K,Z(H))$, so that in particular $\im\mu$ lies in $Z(H)$. We want to show that $\lambda+\mu$ is bijective. Put $\theta=\lambda^{-1}\mu$ and let first $h$ be an element of $H$ such that $\lambda+\mu(h)=1$. This amounts to say that $h=\theta(h^{-1})$. Then $h=\theta^2(h)$ and hence by the incompatibility hypothesis $h=1$. So $\lambda+\mu$ is injective. Let now $h$ be any element of $H$ and let $n$ be the least positive integer such that $\theta^n\lambda^{-1}(h)=1$ and let $$s=\lambda^{-1}(h)\theta\lambda^{-1}(h^{-1})\cdots\theta^{n-1}\lambda^{-1}(h^{(-1)^{n-1}}).$$ As it is easy to see that $\lambda+\mu(s)=h$, we have that $\lambda+\mu$ is also surjective. Since the same also holds for $K$, $\mathcal{A}_{H,K}$ is a subgroup of $\aut (H\times K)$ by Proposition \ref{propsub}.

Conversely, let $\varphi=\left(\begin{array}{cc}\alpha &\beta \\ \gamma & \delta\end{array}\right)$ be an automorphism of $H\times K$, whose inverse is denoted with $\left(\begin{array}{cc}\alpha' &\beta' \\ \gamma' & \delta'\end{array}\right)$. Put $\theta=\beta'\gamma$. Let first $h$ be an element of $H$ such that $\alpha(h)=1$. By Lemma \ref{cases} (2.1), we may write $h=\alpha'\alpha(h)\theta(h)=\theta(h)$. Using Lemma \ref{norm} (2), wee see that $\theta$ is a normal endomorphism of $H$ and thus we can use the total incompatibility of $H$ and $K$ to get $h=1$. Then $\alpha$ is injective and, symmetrically, so are $\alpha'$, $\delta$ and $\delta'$. In particular, $\alpha'\alpha$ is injective. Let $a$ and $b$ be elements of $H$. Since $1-\theta=\alpha'\alpha$ is a homomorphism, we have that $$a^{-1}b\theta(b^{-1})\theta(a)=a^{-1}\theta(a)b\theta(b^{-1}),$$ from which we get $b\theta(b^{-1})=(b\theta(b^{-1}))^{\theta(a)}=b^{\theta(a)}\theta(b^{-\theta(a)})$, using the fact that $\theta$ is a normal endomorphism. In other words, $1-\theta(b)=1-\theta(b^{\theta(a)})$. However, $1-\theta$ is injective, so that $b=b^{\theta(a)}$, which, by the generality of $a$ and $b$, yields that $\im\theta$ is central in $H$. In particular, also $-\theta$ is an endomorphism of $H$. Now we may employ the same proof in the first paragraph to get that $1-\theta$ is bijective and this implies that $\alpha$ is surjective. In the same way, one shows that also $\delta$ is bijective and from Lemma \ref{norm} (1) we finally obtain that $\varphi$ is an element of 
$\mathcal{A}_{H,K}$. We have shown that $\aut (H\times K)\simeq\mathcal{A}_{H,K}$.
\end{proof}

Although the class of totally incompatible pairs is a proper subclass of that of incompatible pairs (see Example \ref{ExTotInc}, Example \ref{ExTotIncPer} and Example \ref{Ex5}), it is still more general than the class of incompatible groups satisfying both the maximal and the minimal conditions on normal subgroups. Indeed, it is not difficult to show that incompatible groups satisfying the maximal and the minimal conditions on normal subgroups are totally incompatible and that the first is a proper subclass of the second one. For instance, the pair $(H,K)$, where $H$ is the additive group of rational numbers and $K$ has order $2$, is an example of a totally incompatible group pair in which not both $H$ and $K$ satisfy the maximal and the minimal conditions on normal subgroups.

\section{Some Applications}\label{SecAppl}

In this section we inspect some group theoretical applications of the results proved so far. Taking inspiration by Lemma \ref{LemComm}, we start looking for conditions for which having no common non-trivial direct factors can tell something about incompatibility. This will soon be shown to be the case of groups satisfying the maximal and the minimal conditions on normal subgroups.

\subsection{Groups satisfying Max-n and Min-n}\label{subsmaxminn}

We first need the following result which is due in its essence to Hans Fitting (see, for instance, \cite[3.3.4]{RobCour}).

\begin{lemma}\label{Fitting}
%Un giorno, in un'altra vita, dimostrare queste cose ma con le ipotesi su un solo gruppo, forse.
Let $G$ be a group satisfying the maximal and the minimal conditions on normal subgroups and let $\varphi$ be a normal endomorphism of $G$. Then there exists a positive integer $r$ such that $\im\varphi^r=\im\varphi^{r+1}=\cdots$, $\ke\varphi^r=\ke\varphi^{r+1}=\cdots$. Moreover, $G=(\im\varphi^r)\times(\ke\varphi^r)$.
\end{lemma}

With the aid of Fitting's lemma, we can show that for groups satisfying the maximal and the minimal condition on normal subgroups the two aforementioned properties are equivalent.

%Let $H$ and $K$ be groups and let $\alpha\in\Hom(H,K)$ and $\beta\in\Hom(K,H)$. For short, we will say that the couple $(\alpha,\beta)$ \textit{fixes some points} if there exist either a non trivial element $h$ of $H$ such that $\beta\alpha(h)=h$ or a non trivial element $k$ of $K$ such that $\alpha\beta(k)=k$.

\begin{theorem}\label{ThComp}
Let $H$ and $K$ be groups satisfying the maximal and the minimal conditions on normal subgroups. Then $H$ and $K$ are incompatible if and only if they have no common non-trivial direct factor.
\end{theorem}
\begin{proof}
First, assume without loss of generality that we may find elements $\alpha$ of $\Hom(H,K)$ and $\beta$ of $\Hom(K,H)$ such that $\alpha\beta$ and $\beta\alpha$ are normal. By Lemma \ref{Fitting}, there exist a positive integer $r$ such that
$$\im(\alpha\beta)^r=\im(\alpha\beta)^{r+1}=\cdots\quad,\quad\ke(\alpha\beta)^r=\ke(\alpha\beta)^{r+1}=\cdots,$$ $$\im(\beta\alpha)^r=\im(\beta\alpha)^{r+1}=\cdots\quad,\quad \ke(\beta\alpha)^r=\ke(\beta\alpha)^{r+1}=\cdots$$ 
%and that$$\im(\beta\alpha)^r=\im(\beta\alpha)^{r+1}=\cdots$$
and that
$$H=\ke\sigma\times\im\sigma,\quad\quad K=\ke\tau\times\im\tau$$ where $\sigma=(\beta\alpha)^r$ and $\tau=(\alpha\beta)^r$. Since $\alpha\sigma=\tau\alpha$, it easily follows that $\alpha$ induces a map from $\im\sigma$ to $\im\tau$. Moreover, let $a$ be an element of $H$ such that $\alpha\sigma(a)=1$. Then $\sigma(a)$ is an element of $\ke\alpha$, which is also a subgroup of $\ke\sigma$, but the latter has trivial intersection with $\im\sigma$ and hence $\sigma(a)=1$. This shows that the map induced by $\alpha$ is injective. In the same way, $\beta$ maps $\im\tau$ to $\im\sigma$ as an injective homomorphism.

Put now $I_\sigma=\im\sigma$ and $I_\tau=\im\tau$, suppose that $H$ contains a non-trivial element $h$ such that $\beta\alpha(h)=h$ and assume for a contradiction that $I_\sigma$ and $I_\tau$ are not isomorphic. Then, in particular, $\alpha(I_\sigma)$ is a proper subgroup of $I_\tau$, because the restriction of $\alpha$ to $I_\sigma$ is injective. Moreover, $\alpha(I_\sigma)$ is not trivial, because $h$ is a non-trivial element of $I_\sigma=\im\sigma$. For the same reason $\beta\alpha(I_\sigma)$ is a proper subgroup of $I_\sigma$. However, $\beta\alpha(I_\sigma)=\im (\beta\alpha)^{r+1}=\im\sigma$, a contradiction.
%Define by recursion $K_{i+1}=\alpha(H_i)$ and $H_{i+1}=\beta(K_{i+1})$ for each natural number $i\geq 1$. We now easily show by induction that every $K_{i+1}$ is a proper subgroup of $K_i$. Assume that $K_{i-1}$ is a proper subgroup of $K_i$ for some natural number $i\geq2$. Since $\beta$ maps injectively $K_{i-1}$ to $H_{i-1}$, $H_i$ must be a proper subgroup of $H_{i-1}$. Moreover, the restriction of $\alpha$ to $H_i$ is injective, too, and hence $K_{i+1}$ is a proper subgroup of $K_i$ for the same reason. On the other hand, for each $K_i$

The converse follows from Lemma \ref{LemComm}.
\end{proof}

Following the proof of Theorem \ref{ThComp}, we get the following case for centrally incompatible pairs of groups, which should be seen in analogy to Proposition \ref{PropCInc}.

\begin{theorem}\label{ThCComp}
Let $H$ and $K$ be groups satisfying the maximal and the minimal conditions on normal subgroups. Then $H$ and $K$ are centrally incompatible if and only if they have no common non-trivial central direct factor.
\end{theorem}

Putting together Proposition \ref{PropCInc} and Theorem \ref{ThCComp}, the following is straightforward.

\begin{theorem}\label{ThCCompAut}
Let $H$ and $K$ be groups satisfying the maximal and the minimal conditions on normal subgroups. Then $\mathcal{A}_{H,K}$ is a subgroup of $\aut (H\times K)$ if and only if they have no common non-trivial central direct factor.
\end{theorem}

On the other hand, putting together Theorem \ref{incompatible} and Theorem \ref{ThComp}, we get the following

\begin{theorem}\label{ThCompAut}
Let $H$ and $K$ be groups satisfying the maximal and the minimal conditions on normal subgroups. Then $\aut (H\times K)\simeq\mathcal{A}_{H,K}$ if and only if they have no common non-trivial direct factor.
\end{theorem}

The next result has to be seen in connection with Corollary \ref{cor1}.

\begin{cor}\label{corr}
Let $H$ and $K$ be groups satisfying the maximal and the minimal conditions on normal subgroups. If $H$ and $K$ have no common non-trivial direct factor and both $\Hom(H,Z(K))$ and $\Hom(K,Z(H))$ are zero, then $\aut (H\times K)\simeq\aut H\times\aut K$.
\end{cor}

Finally, from Theorem \ref{ThCompAut} we obtain also the following known result.

\begin{cor}\label{corrr}
Let $H$ and $K$ be finitely generated abelian groups. Then $\aut (H\times K)\simeq\mathcal{A}_{H,K}$ if and only if they have no common non-trivial direct factor.
\end{cor}
\begin{proof}
The statement immediately follows from the structure of finitely generated abelian groups, which yields in particular that $H$ and $K$ have no common non-trivial direct factor if and only if they are both finite.
\end{proof}

Notice here that Example \ref{Ex1} shows that Corollary \ref{corrr} cannot be extended to infinitely generated abelian groups, even if they have finite $0$-rank.

\bigskip

The aim is now to go further beyond the class of groups satisfying the maximal and the minimal conditions on normal subgroups. We begin applying the previous methods to two antipodal classes of groups: abelian groups and stem groups.

\subsection{Abelian groups}

\begin{lemma}\label{AbIncTotInc}
Let $H$ and $K$ be abelian groups. If $H=D\oplus B\oplus F$, where $D$ is divisible, $B$ has finite exponent and $F$ is free, then $H$ and $K$ are incompatible if and only if $H$ and $K$ are totally incompatible.
\end{lemma}
\begin{proof}
We just have to prove necessity. Assume that $H$ and $K$ are incompatible. Then $H$ and $K$ have no non-trivial direct factor in common by Lemma \ref{LemComm} and in particular if $H$ is infinite, $K$ has no infinite cyclic quotient. Hence, $\Hom(K,H)=\Hom(K,D\oplus B)$ and we may assume that $F=0$. Let now $\psi$ and $\chi$ be elements of $\Hom(H,K)$ and of $\Hom(K,H)$, respectively, and assume for a contradiction that $H$ contains an element $h$ such that $(\chi\psi)^l(h)\neq0$ for any positive integer $l$. From this we may assume without loss of generality that all the elements of the form $\psi^l(\chi\psi)^m(h)$ for any non-negative integers $l$ and $m$ have the same order.
%La successione si deve stabilizzare, altrimenti va a 0
Since divisible subgroups of abelian groups are always direct factors and $\psi(D)$ is divisible, it follows that $\Hom(D,\psi(D))=0$ and hence $h$ is not an element of $D$. So we can assume without loss of generality that $H$ has finite exponent, call it $e$. 
%In pratica sostituiamo $H$ con il primo zoccolo che contiene h
Let now $a_1$ be an element of $H$ such that $p^na_1=h$ for a non-negative integer $n$ and such that the equation $px=a_1$ has no solution in $H$. By definition, $\gen{a_1}$ is a pure bounded subgroup of $H$ and hence it is a direct summand (see for instance 4.3.8 in \cite{RobCour}). In the same way, take a direct summand $\gen{b_1}$ of $K$ containing $\psi(a_1)$. Since $h$ and $\psi(h)$ have the same order, $\psi$ has to map injectively $\gen{a_1}$ to $\gen{b_1}$, so that $|\gen{a_1}|<|\gen{b_1}|$, where the inequality comes from the fact that $H$ and $K$ have no common non-trivial direct summands. Let now $\gen{a_2}$ be a direct summand of $H$ containing $\chi\psi(h)$. For the same reason, $|\gen{b_1}|<|\gen{a_2}|$ and in particular $|\gen{a_1}|<|\gen{a_2}|$. However, $H$ has finite exponent and hence repeating this argument leads to a contradiction.
\end{proof}

Examples \ref{ExTotInc} and \ref{ExTotIncPer} show that the same cannot be said for abelian groups in general, be they torsion-free or periodic.

\begin{theorem}\label{PropAb} Let $H$ and $K$ be abelian groups. If $H=D\oplus B\oplus F$, where $D$ is divisible, $B$ has finite exponent and $F$ is free, then $\aut (H\oplus K)\simeq\mathcal{A}_{H,K}$ if and only if $H$ and $K$ have no common non-trivial direct summand.
\end{theorem}
\begin{proof}
First, let $H_1$ and $K_1$ be isomorphic non-trivial direct summands of $H$ and $K$, respectively, write $H=H_1\oplus H_2$ and $K=K_1\oplus K_2$ and let $\theta:H_1\mapsto K_1$ be an isomorphism. For any couple of groups $X$ and $Y$ such that $Y\leq X$, let $\iota_{Y,X}$ be the immersion of $Y$ in $X$ and $\pi_{X,Y}$ be the canonical projection of $X$ onto $X/Y$. For the sake of simplicity, put $H_i=H/H_j$ and $K_i=K/K_j$ for $i,j\in\{1,2\}$ and $i\neq j$. Put now
%Li mettevo facendo vedere che Aut non è in A: $\alpha=\iota_{H_2,H}\pi_{H,H_1}$, $\beta=\iota_{H_1,H}\varphi\pi_{K,K_2}$, $\gamma=\iota_{K_1,K}\varphi\pi_{H,H_2}$ and $\delta=\iota_{K_2,K}\pi_{K,K_1}$.
$\beta=\iota_{H_1,H}\theta^{-1}\pi_{K,K_2}$, $\gamma=\iota_{K_1,K}\theta\pi_{H,H_2}$
and take into account the endomorphism of $H\oplus K$ $$\varphi=\left(\begin{array}{cc}1 &-\beta \\ -\gamma & 1\end{array}\right).$$ Let $h$ be a non-trivial element of $H_1$. Therefore, $\varphi(h,\theta(h))=(0,0)$, so that $\mathcal{A}_{H,K}$ is not a subset of $\aut(H\oplus K)$ and this proves the necessary condition.

Assume now that $H$ and $K$ share no non-trivial direct summand. We want to show that $H$ and $K$ are incompatible and hence we may assume that the torsion subgroup of $H$ is a $p$-group for a prime $p$. To this aim, let $\psi$ and $\chi$ be elements of $\Hom(H,K)$ and of $\Hom(K,H)$, respectively, and assume for a contradiction that we may find a non-trivial $h$ of $H$ such that $\chi\psi(h)=h$. Put $k=\psi(h)$. By this condition, it follows in particular that $h$ and $k$ have the same order. Notice moreover that $K$ has no infinite cyclic quotient, so that in particular $h$ and $k$ are torsion elements and we may assume without loss of generality that $F=0$. Since abelian groups always split over divisible subgroups, neither $h$ nor $k$ belong to such subgroups, otherwise $H$ and $K$ would have a common non-trivial direct factor. In particular, we may assume that both $H$ and $K$ are periodic and reduced.
Since $h$ is not contained in a divisible subgroup, we may find a subgroup $\gen{a}$ of $H$ which has maximal order among the cyclic subgroups containing $h$. In particular, $\gen{a}$ is a pure bounded subgroup of $H$ and hence it is a direct summand. Since $h$ and $k$ have the same order, $a$ and $\psi(a)$ have the same order, too. Suppose that $K$ contains an element $b$ such that $pb=\psi(a)$. Since $k$ and $\chi(k)$ have the same order, then also $b$ and $\chi(b)$ have the same order, which means that $|\gen{\chi(b)}|=p|\gen{a}|$ and this goes against the maximality of $|\gen{a}|$. This shows that $\gen{\psi(a)}$ is a pure subgroup of $K$ and hence it is even a direct summand of $K$. However, it is isomorphic with $\gen{a}$, which shares the same property in $H$. This contradiction yields that $h=1$ and hence $H$ and $K$ are incompatible. The thesis now follows from Lemma \ref{AbIncTotInc} and from Proposition \ref{htotincompatible}.
\end{proof}

Example \ref{Ex1} shows that the same cannot be said in general if both $H$ and $K$ are torsion-free abelian groups.

\subsection{Stem groups}\label{SubStem}

Following P. Hall \cite{Hall}, a group is called a \textit{stem group} if its centre lies inside its derived subgroup. Stem groups are importantly related with the classification of finite $p$-groups given by Hall himself and with groups extensions. As far as we are concerned, stem groups are directly linked with central compatibility in the following sense: Let $H$ be a group, let $K$ be a stem group, let $\alpha\in\Hom(H,Z(K))$ and let $\beta\in\Hom(K,Z(H))$. Then it is clear that $\alpha\beta$ is trivial, so that in particular $H$ and $K$ are centrally incompatible. More generally, one could say that if one between $H$ and $K$ is a stem group, then the pair $(H,K)$ is centrally totally incompatible of length $1$, meaning that for every $\alpha\in\Hom(H,Z(K))$ and $\beta\in\Hom(K,Z(H))$ either $\alpha\beta=0$ or $\beta\alpha=0$. However, this goes beyond the scopes of this paper.

In analogy with the abelian case portrayed above, one may also ask if for any pair of stem groups, being incompatible and totally incompatible are equivalent. In general this is not the case, as it is shown by Example \ref{Ex5}. However, it will be now proved that in case one between $H$ and $K$ is a stem group, the conditions on normal subgroups of $H\times K$ can be removed, so that we will not need the pair to be totally incompatible to get good results on $\aut(H\times K)$.

\begin{prop}\label{PropStem} Let $H$ and $K$ be groups. If $H$ is a stem group, then $\mathcal{A}_{H,K}$ is a subgroup of $\aut (H\times K)$.
\end{prop}
\begin{proof}
Take $\lambda\in\aut H$, $\mu\in \Hom(H,Z(K))$, $\nu\in\aut K$ and $\xi\in \Hom(K,Z(H))$ Since $\im\xi\leq Z(H)\leq H'\leq \ke\mu$, we have that $\mu\xi=0$. Notice that this holds for any couple of such homomorphisms and in particular also for $\mu(\nu^{-1}\xi)$. If $x$ is an element of $K$ such that $\nu-\xi\mu(x)=1$, then $x=\nu^{-1}\xi\mu(x)=(\nu^{-1}\xi\mu)^2(x)=(\nu^{-1}\xi(\mu\nu^{-1}\xi)\mu)(x)=1$ and hence $\nu-\xi\mu$ is injective. On the other hand, if $y$ is any element of $K$, then it is straightforward to check that $$\nu-\xi\mu(\nu^{-1}\xi\mu\nu^{-1}(y)\nu^{-1}(y))=y,$$ so $\nu-\xi\mu$ is bijective. As clearly $\lambda+\mu\xi=\lambda$,we may apply Proposition \ref{propsub} to get our thesis.
\end{proof}

We are now ready to prove the following theorem, which shows that if one between $H$ and $K$ is a stem group (and obviously if $H$ and $K$ share no non-trivial direct factor), then $\mathcal{A}_{H,K}$ is the whole group of automorphisms of $H\times K$. In particular, the determinant is defined for every automorphism of the direct products of two groups, provided that at least one of them is a stem group and that they share no non-trivial direct factor.

\begin{theorem}\label{stem} Let $H$ and $K$ be groups. If $H$ is a stem group, then $\aut(H\times K)\simeq \mathcal{A}_{H,K}$ if and only if $H$ and $K$ have no common non-trivial direct factor.
\end{theorem}
\begin{proof}
Necessity follows from Lemma \ref{LemComm}, so assume that $H$ and $K$ share no common direct factor. Since Proposition \ref{PropStem} gives us that $\mathcal{A}_{H,K}\subseteq\aut (H\times K)$, we just have to prove the reverse inclusion. Let hence $$\varphi=\left(\begin{array}{cc}\alpha &\beta \\ \gamma & \delta\end{array}\right)\quad\text{and}\quad\varphi^{-1}=\left(\begin{array}{cc}\alpha' &\beta' \\ \gamma' & \delta'\end{array}\right)$$ be automorphisms of $H\times K$ and put $\sigma=\alpha\alpha'$, $\tau=\beta\gamma'$, $I_\sigma=\im\sigma$, $I_\tau=\im\tau$ and $I=I_\sigma\cap I_\tau$. The following two facts, respectively derived from Lemma \ref{cases} (1.1) and Proposition \ref{end}, will be used without further reference: $1=\sigma+\tau$ and $[I_\sigma, I_\tau]=1$. First, we immediately get that $H=I_\sigma I_\tau$ and that $I$ is central. If $h=\tau(x)$ is any element of $I_\tau$, then $\tau(x)=\sigma(h)\tau(h)$, so that $\sigma(h)=\tau(xh^{-1})$ and $\sigma(I_\tau)$ is contained in $I$. Symmetrically the same holds for $\tau(I_\sigma)$. As $I$ is central, it follows in particular that $I'_\sigma\leq \ke\tau$ and $I'_\tau\leq\ke\sigma$. Let $x$ be any element of $H$. Then $\sigma(x)=\sigma^2(x)\tau\sigma(x)$. On the other hand, $\sigma(x)=\sigma(\sigma(x)\tau(x))=\sigma^2(x)\sigma\tau(x)$, so that we have $\sigma\tau=\tau\sigma$.
%In particular, every element of  $I=\im\sigma\tau$?Utile?Dimostrarlo?.
Also, $I\leq Z(H)=H'=I'_\sigma I'_\tau\leq \ke\tau\ke\sigma$.
% From this it follows that $(\sigma\tau)^2=0$.?Utile?Dimostrarlo?

Notice here that, for any $h$ in $\ke\sigma\cap\ke\tau$, $h=\sigma(h)\tau(h)=1$, so that $\ke\sigma\cap\ke\tau=1$. Put now $I_1=I'_\tau\cap I_\sigma$ and $I_2=I'_\sigma\cap I_\tau$. From what just shown, it follows in particular that $I_1\cap I_2=1$. Since $I\leq I'_\sigma I'_\tau$, by Dedekind's Law we have that $I=I\cap I'_\sigma I'_\tau=I_\sigma\cap(I_\tau\cap I'_\sigma I'_\tau)=I_\sigma\cap (I_2I'_\tau)=I_1\times I_2$. Take now into account the restriction of $\sigma^2$ to $I_\tau$. By the structure of $I$, we have that the image of $\sigma^2|_{I_\tau}$ lies inside $I_2$. If $x,y$ and $z$ are elements of $H$ such that $\sigma([x,y])=\tau(z)$, then $\tau(z)$ is an element of $I'_\sigma\leq\ke\tau$ and hence $\tau(z)=\sigma\tau(z)\tau^2(z)=\sigma\tau(z)=\sigma^2([x,y])$. This shows that $\sigma^2_{|I_\tau}$ acts as the identity on $I_2$, which is hence its image. Therefore, $I_\tau$ splits over $I_2$ and we may write $I_\tau=H_2\times I_2$, where $H_2=\ke \sigma^2_{|I_\tau}$. Making use of $\tau^2_{|I_\sigma}$, we similarly write $I_\sigma=H_1\times I_1$, where $H_1=\ke\tau^2_{|I_\sigma}$. It follows in particular that $I'_\sigma=H_1'$, that $I'_\tau=H_2'$ and that $H=H_1\times H_2$.
%Metto I2 in H1 e I1 in H2
Also, if $h$ is an element of $\ke\tau\cap H_2$, then $h=\sigma(h)$ and so $h=\sigma^2(h)=1$. On the other hand, as $\sigma$ and $\tau$ commute, for any $x$ in $H_2$ we have that $\sigma^2\tau(x)=\tau\sigma^2(x)=1$.
%Cioè la suriettività
Therefore, $\tau$ acts as a bijection on $H_2$. Moreover, as $H$ is a stem group, both $H_1$ and $H_2$ are stem groups. Assume now for a contradiction that $\alpha$ is not surjective, so that neither $\sigma$ is such. Then $I_\tau$ is not contained in $I_\sigma$ and hence $H_2$ is not trivial. Put $\gamma'(H_2)=K_1$, $\phi:x\in H_2\mapsto\gamma'(x)\in K_1$ and $\chi:x\in K_1\mapsto\beta(x)\in H_2$. The fact that $\tau$ acts as a bijection on $H_2$ is equivalent to say that $\chi\phi$ is bijective. On the other hand, also $\phi$ is bijective and hence so is $\chi$. Let $\pi$ be the canonical projection of $H$ onto $H_2$. Then $\pi\beta(K)=H_2$
%Perché tanto già sappiamo che \beta(K_1) ricopre H_2
and $\gamma'\pi\beta$ is an epimorphism of $K$ onto $K_1$, while clearly $(\gamma'\pi\beta)(\phi\chi)^{-1}=1_{|K_1}$. This provides a splitting of $K=K_1\ltimes K_2$ for a subgroup $K_2$ of $K$. Now, using the facts that $\gamma'\beta+\delta'\delta=1$ and that $[\im\gamma',\im\delta']=1$, we also get that $K_1$ is a direct factor of $K$. However, $K_1$ is isomorphic to $H_2$ because $\phi$ is bijective and this is not possible. This contradiction arises from our assumption on $\alpha$ being not surjective. Thus $\alpha$ is surjective. In particular, again from the structure of $\mathcal{M}_{H,K}$, $\beta$ belongs to $\Hom(K,Z(H))$. Then $I_\tau\leq Z(H)$ and hence $\tau^2=0$ since $Z(H)\leq H'$. Moreover, the exact same proof with $\alpha'\alpha$ and $\beta'\gamma$ in place of $\sigma$ and $\tau$, respectively, can be employed to show that also $(\beta'\gamma)^2=0$. If now $h$ is an element of $\ke\alpha$, then $h=\alpha'\alpha(h)\beta'\gamma(h)=\beta'\gamma(h)=(\beta'\gamma)^2(h)=1$. Thence $\alpha$ is an automorphism. 

Finally, let $\psi=\gamma'\beta$. Since $\tau^2=0$, we get that $\psi^3=0$. Now, $\delta'\delta=1-\psi$. Let $k$ be an element of $\ke\delta'\delta$. Then $k=\psi(k)=\psi^2(k)=\psi^3(k)=1$, so that $\delta'\delta$ is injective. On the other hand, if $k$ is any element of $K$, it is immediate to see that $1-\psi(k\psi(k)\psi^2(k))=k$ and hence $\delta'\delta$ is an automorphism of $K$. As also $(\beta'\gamma)^2=0$, then $(\gamma\beta')^3=0$, so that, by the same argument, we get that $\delta\delta'$ is bijective. It follows that also $\delta$ is bijective and hence, in particular, $\gamma$ is an element of $\Hom(H,Z(K))$. This shows that $\varphi$ is an element of $\mathcal{A}_{H,K}$ and our proof is concluded.
\end{proof}

\begin{theorem}\label{ThCompStem}
Let $H$ and $K$ be stem groups with no common non-trivial direct factor. Then $$\aut (H\times K)\simeq(\aut H\times\aut K)\ltimes(\Hom(H,Z(K))\times\Hom(K,Z(H))).$$
\end{theorem}
\begin{proof}
First, $\aut(H\times K)$ is isomorphic with $\mathcal{A}_{H,K}$ by Theorem \ref{stem}. Let $U$ and $L$ be the subgroups of upper and lower unitriangular matrices of $\mathcal{A}_{H,K}$, respectively. From Proposition \ref{AStruc}, we have that diag$\mathcal{A}_{H,K}\simeq \aut H\times\aut K$ normalizes both $U$ and $L$, so the only thing to show is that $[U,L]=1$. To this aim, let $\varphi=\left(\begin{array}{cc}1 &\beta \\ 0 & 1\end{array}\right)$ and $\chi=\left(\begin{array}{cc}1 &0 \\ \gamma & 1\end{array}\right)$ be two elements of $\mathcal{A}_{H,K}$. Since by hypothesis $Z(H)\leq H'\leq \ke\gamma$ and $Z(K)\leq K'\leq\ke\beta$, we have that both $\beta\gamma$ and $\gamma\beta$ are zero homomorphisms and this shows that $[\varphi,\chi]=1$, which yields the thesis.
\end{proof}

We remark here that the condition that both $H$ and $K$ are stem groups cannot be removed from last theorem. Explicitly, if $H$ is a quaternion group of order $8$, $K$ is quotient of order $2$ of $H$, $\beta$ the monomorphism of $K$ into $Z(H)$ and $\gamma$ is the canonical projection of $H$ onto $K$, then the automorphisms $\left(\begin{array}{cc}1 &\beta \\ 0 & 1\end{array}\right)$ and $\left(\begin{array}{cc}1 &0 \\ \gamma & 1\end{array}\right)$ of $H\times K$ are easily seen not to commute.

\subsection{Central automorphisms}

Let $G$ be a group. An automorphism of $G$ is said to be \textit{central} if it acts as the identity over $G/Z(G)$. The subgroup of central automorphisms of $G$ is usually denoted with $\aut_c(G)$. Central automorphisms are widely studied and it is still an open question how to usefully characterise the structure of (at least finite) groups all of whose automorphisms are central. Here we give a description of $\aut_c(G)$ in case $G$ is a direct product of two groups of the kind studied in some previous result.

To this aim, for any couple of groups $H$ and $K$ we define the following set (see \cite{BidCurMcC}) $$\mathcal{Z}_{H,K}=\left\{\left(\begin{array}{cc}\alpha &\beta \\ \gamma & \delta\end{array}\right):\begin{array}{ll} \alpha\in \aut_c H, &\beta\in \Hom(K,Z(H)) \\
\gamma\in \Hom(H,Z(K)), &\delta\in \aut_c K
\end{array}
\right\}$$

\begin{cor}\label{ThCompAutc}
Let $H$ and $K$ be groups satisfying one of the following conditions:
\begin{itemize}
\item[(1)] $H$ and $K$ are abelian and $H=D\oplus B\oplus F$, where $D$ is divisible, $B$ has finite exponent and $F$ is free;
\item[(2)] $H$ or $K$ is a stem group;
\item[(3)] $H$ and $K$ satisfy the maximal and the minimal conditions on normal subgroups.
\end{itemize}
If $H$ and $K$ have no common non-trivial direct factors, then $\aut_c (H\times K)\simeq\mathcal{Z}_{H,K}$.
\end{cor}
\begin{proof}
Let $\varphi$ be an element of $\aut_c (H\times K)$. By either Theorem \ref{ThCompAut} or Theorem \ref{PropAb} or Theorem \ref{stem}, we have that $\aut(H\times K)\simeq \mathcal{A}_{H,K}$, so, if we take $\varphi=\left(\begin{array}{cc}\alpha &\beta \\ \gamma & \delta\end{array}\right)$ to be an automorphism of $H\times K$, we only have to show that $\alpha$ and $\delta$ are central automorphisms of $H$ and $K$, respectively. Also, we can show this just for $\alpha$, as for $\delta$ the proof is completely analogous. If $(h,1)$ is an element of $H\times K$, we have that $$(h,1)(Z(H)\times Z(K))=\varphi((h,1)(Z(H)\times Z(K)))=\alpha(h)Z(H)\times\gamma(h)Z(K)$$ and hence $\alpha(hZ(H))=hZ(H)$.
\end{proof}

\medskip

\subsection{Computing inverse automorphisms}\label{GAP}

Let $H$ and $K$ be finite groups of order $m$ and $n$, respectively, and assume that $m\leq n$. Let $\varphi$ be an endomorphism of $H\times K$. One frequent task is to check whether $\varphi$ is bijective (i.e. injective) or not. A simple algorithm to check injectivity for $\varphi$ would require $\binom{mn}{2}$ steps, in which we check if the image of each element of $H\times K$ via $\varphi$ is a value already taken by the previous ones.

Now, suppose that $H$ and $K$ have no common non-trivial direct factor and assume that we don't now anything about $\varphi$ except for its components, which may be defined by $\alpha_\varphi=\pi_H\varphi\iota_H$, $\beta_\varphi=\pi_H\varphi\iota_K$, $\gamma_\varphi=\pi_K\varphi\iota_H$ and $\delta_\varphi=\pi_K\varphi\iota_K$ (see Proposition \ref{end}). Now, Theorem \ref{ThCompAut} ensures that $\aut(H\times K)$ is isomorphic with $\mathcal{A}_{H,K}$, so that both $\alpha$ and $\delta$ are bijective. Calculate for instance $\delta^{-1}$, a task which can be accomplished in $n$ steps simply swapping coordinates inside its graph, and check if $\det\varphi=\alpha-\beta\delta^{-1}\gamma$, which is a map from $H$ to $H$, is invertible. Then we have found that $\varphi$ is or is not invertible in at most $n+\binom{m}{2}$. Notice that we could also calculate $\alpha^{-1}$ and $(\det_K\varphi)^{-1}$ and this would take $m+\binom{n}{2}$ steps. However, since $m\leq n$, then $m+\binom{n}{2}\leq n+\binom{m}{2}$ only in case $n=1,2$, so that looking for $\varphi^{-1}$ via $\det\varphi$ is a more convenient choice if we assume the not-so-restrictive hypothesis that $|H|>2$. Now Theorem \ref{propinv} yields that $\varphi$ is bijective and this is performed in an almost always better way than checking this on $\varphi$, since $\binom{mn}{2}\leq n+\binom{m}{2}$ only in case $m=1$. Finally, the inverse of $\varphi$ can be computed using Proposition \ref{invgen}. Then the following result is computationally useful.

\begin{prop} Let $H$ and $K$ be finite groups with no common non-trivial direct factor such that $|H|=m\leq n=|K|$ and let $\varphi$ be an endomorphism of $H\times K$. Then $\varphi$ can be checked to be bijective in $n+\binom{m}{2}$ steps.
\end{prop}

If we do not already know whether $H$ and $K$ have some common non-trivial direct factor, we can still use the above algorithm, which will be better of the usual one in case either $\alpha$ or $\delta$ is invertible, but which will give no answer otherwise.

\medskip

As for the infinite case, something computationally useful can be said. Let $H$ and $K$ be groups, let $$\varphi=\left(\begin{array}{cc}\alpha &\beta \\ \gamma & \delta\end{array}\right)$$ be an endomorphism of $H\times K$ and assume that $K$ is infinite. Assume moreover that one has some further information on $\delta$, which allows to find the inverse $\delta^{-1}$ of $\delta$. This can happen if, for instance, $\aut K$ has a restricted structure or if $\delta$ just came with some specific hypotheses. In this case, to infer whether $\varphi$ is invertible or not it is sufficient to calculate the inverse of $\det_H\varphi$, which is a function on a group which is maybe more manageable than $K$. Also, if $\det_H\varphi$ is invertible, Proposition \ref{invgen} makes it possible to express $\varphi$ in terms of $\alpha,\beta,\gamma,\delta$ and $\det_H\varphi$. As an example, we have the following

\begin{prop} Let $H$ and $K$ be groups, let $H$ be finite of order $m$ and let $\varphi$ be an endomorphism of $H\times K$. If $\pi_K\varphi\iota_K$ is invertible, then $\varphi$ can be checked to be bijective in $\binom{m}{2}$ steps.
\end{prop}

These are only some of the possible applications of the determinants. For instance, a more detailed study of incompatible and totally incompatible group pairs can lead to extend the range of application.

\section{Examples}\label{SecEx}

Here we present several examples, whose aims are to show that our results are optimal, in a sense, and to present some instances of application of the determinants.

\subsection{$\mathcal{A}_{H,K}\not\subseteq \aut(H\times K)$}

We begin to show that Theorem \ref{ThCCompAut} does not hold in general. Namely, that $\mathcal{A}_{H,K}$ is not contained in $\aut(H\times K)$ even if $H$ and $K$ have no common non-trivial direct factor. In particular, the same counterexample shows also that Theorem \ref{ThCompAut} does not hold in general. The groups in the following example are borrowed from a well-known class of indecomposable torsion-free abelian groups (see, for instance, \cite[Lemma 88.3]{fuch2}).

\begin{example}\label{Ex1} There exist two torsion-free abelian groups of finite rank $H$ and $K$ with no common non-trivial direct summand such that $\mathcal{A}_{H,K}$ is not contained in $Aut(H\oplus K)$. Moreover, in general $\mathcal{A}_{H,K}$ need not be even a stable part of the monoid $(\mathcal{M}_{H,K},\cdot)$.
\end{example}
\begin{proof}
Let $p,q,r$ and $s$ be different positive prime numbers. Let $\{x,y\}$ be a basis of $\mathbb{Q}\oplus\mathbb{Q}$ and let $H$ and $K$ be the subgroups of $\mathbb{Q}\oplus\mathbb{Q}$ generated respectively by the following subsets $$\{sp^lx,q^my,r^n(x+y)\mid l,m,n\in\mathbb{Z}\}\quad\text{and}\quad\{p^lx,q^my,r^n(x+y)\mid l,m,n\in\mathbb{Z}\}.$$ Clearly, $H$ is a subgroup of $K$. Let $\theta$ be a non-zero homomorphism from $K$ to $H$. Since every torsion-free non-trivial quotient of $K$ has some non-trivial elements with infinite $p_1$- and $p_2$-height for two distinct prime numbers $p_1$ and $p_2$ and $H$ has no such elements, we have that $\theta$ is injective. Now $\theta$ acts preserving the heights, so that we may find three integers $a,b,c$ and positive integers $l,m$ and $n$ such that $\theta(x)=sp^alx$, $\theta(y)=q^bmy$ and $\theta(x+y)=r^cn(x+y)$, from which it follows that $sp^al=q^bm=r^cn$. However, $p$, $q$, $r$ and $s$ are different, so that there exists a positive integer $d$ such that $\theta:k\in K\mapsto sdk\in H$. Let now $h^K_s$ be the $s$-height function over $K$ and regard $H$ as a subgroup of $K$. As no element of $K$ has infinite $s$-height, for each non-trivial element $k$ of $K$ we have that $h^K_s(k)<h^K_s(\theta(k))$. Notice that from the generality of $\theta$, it follows in particular that $H$ and $K$ are not isomorphic. Furthermore, as $H$ is easily seen to be fully invariant in $K$,
%Sempre dalle relazioni di sopra, però da H a K
if we denote with $h^H_s$ the $s$-height function over $H$, every homomorphism $\eta$ from $H$ to $K$ is such that $h^H_s(h)\leq h^H_s(\eta(h))$ for every element $h$ of $H$. This immediately implies that $K$ and $H$ are incompatible, since for every $\chi$ in $\Hom(H,K)$, $\psi$ in $\Hom(K,H)$ and $k\in K$, $h^K_s(k)<h^K_s(\chi\psi(k))$. Moreover, incompatible groups share no common non-trivial direct summand by Lemma \ref{LemComm}, so this is the case for $H$ and $K$.

Assume now the further conditions that $s$ is odd and that $s\not\equiv_p1$. Let $\gamma$ be the immersion of $H$ in $K$, let $\beta$ be the monomorphism which maps $K$ to $sK$ inside $H$ and take into account the following element of $\mathcal{A}_{H,K}$: 
$$\varphi=\left(\begin{array}{cc}1 &\beta \\ \gamma & 1\end{array}\right).$$ If we prove that the endomorphism $1-\gamma\beta$ of $K$ is not surjective, then we are done by applying Theorem \ref{propinv}. Assume the existence of an element $k$ of $K$ such that $k-\gamma\beta(k)=x$ and call $p^dex$ the $x$-component of $k$, for an integer $d$ and a non-zero integer $e$. Then $p^dex-sp^dex=x$, so that $(1-s)p^de=1$. Since $s$ is odd, $d$ cannot be $0$ and this implies that $p$ divides $1-s$, which is against our other assumption on $s$.
%Detto esplicito che non è suriettiva Let now suppose that we may find a couple $(h,k)$ in $H\times K$ such that $\varphi(h,k)=(0,x)$. Then $h+\beta(k)=0$ and $\gamma(h)+k=x$, from which it follows that $k=x+\gamma\beta(k)$ and has been shown to be impossible.
Then $\varphi$ is not an automorphism of $H\times K$ and hence $\mathcal{A}_{H,K}$ contains elements which are not automorphisms of $H\times K$.

Finally, $$\left(\begin{array}{cc}1 &-\beta \\ \gamma & 1\end{array}\right)^2=\left(\begin{array}{cc}* &* \\ * & 1-\gamma\beta\end{array}\right)$$
does not belong to $\mathcal{A}_{H,K}$, as $1-\gamma\beta$ is not an automorphism of $K$. This shows that $\mathcal{A}_{H,K}$ need not be closed under row-by-column multiplication.
\end{proof}

From the previous example we can get also the following information, which can be proved either directly or also by an application of Proposition \ref{htotincompatible}.

\begin{example}\label{ExTotInc}
There exist two torsion-free abelian groups $H$ and $K$ which are incompatible and not totally incompatible.
\end{example}

As these groups are incompatible and $\mathcal{A}_{H,K}$ is not a subset of $Aut(H\times K)$, they can also seen in connection with Proposition \ref{PropCInc}.

\medskip

Incompatibility and total incompatibility are not equivalent also in the class of periodic abelian groups. To show this we have the following

\begin{example}\label{ExTotIncPer}
There exist two periodic abelian groups $H$ and $K$ with no common non-trivial direct summand such that $\mathcal{A}_{H,K}$ is not contained in $Aut(H\oplus K)$. Moreover, $H$ and $K$ are incompatible and not totally incompatible.
\end{example}
\begin{proof}
Let $p$ be a prime and let $$H=\bigoplus\limits_{i\in\mathbb{N}\setminus\{0\}}\gen{h_i}\quad\text{ and }\quad K=\bigoplus\limits_{i\in\mathbb{N}\setminus\{0\}}\gen{k_i}$$ where, for each $i\in\mathbb{N}\setminus\{0\}$, $|\gen{h_i}|=p^{2i-1}$ and $|\gen{k_{i}}|=p^{2i}$. Let $\psi$ and $\chi$ be elements of $\Hom(H,K)$ and of $\Hom(K,H)$, respectively, and assume for a contradiction the existence of two non-trivial elements $h$ and $k$ of $H$ and $K$, respectively, such that $\psi(h)=k$ and $\chi(k)=h$. By this property, $h$ and $k$ has the same order. As $H$ has no elements of infinite height, we may find a cyclic direct summand $\gen{a}$ of $H$ containing $h$ of maximal order, call it $p^m$ for a positive integer $m$ which is odd by the structure of $H$. As also $K$ has no elements of infinite height, we find a cyclic direct summand of $K$ containing $\psi(a)$, which has order $p^n$ for some even positive integer. Notice that, since $h$ and $k$ have the same order, $n$ must be strictly greater than $m$. By the same reason, $\chi\psi(a)$ must be contained in a direct summand of $H$ containing $h$, which has order strictly greater than $p^n$ and this is impossible. Therefore, $H$ and $K$ are incompatible.

Now we show that $H$ and $K$ are not totally incompatible. To this aim let $\gamma$ be defined by the assignments $h_i\in H\mapsto pk_i\in K$ for every $i\in\mathbb{N}$, let $\beta$ be defined by $k_i\in K\mapsto ph_{i+1}\in H$ for every $i\in\mathbb{N}$ and consider the following element of $\mathcal{A}_{H,K}$: 
$$\varphi=\left(\begin{array}{cc}1 &\beta \\ \gamma & 1\end{array}\right).$$

As, for instance, $h_1$ cannot be an element of the image of $1-\beta\gamma$, it follows that the latter is not surjective. Then Theorem \ref{propinv} yields that $\varphi$ is not invertible and $\mathcal{A}_{H,K}$ is not contained in $Aut(H\times K)$. Finally, one can prove that $H$ and $K$ are not totally incompatible by inspecting $\beta\gamma$ or also by an application of Proposition \ref{htotincompatible}.
\end{proof}

\subsection{$\aut(H\times K)\not\subseteq \mathcal{A}_{H,K}$}

We now turn to the reverse inclusion. We show that, differently from what stated for instance in Theorem \ref{ThCompAut} or in Theorem \ref{PropAb}, if $H$ and $K$ are groups $Aut(H\times K)$ need not be contained in $\mathcal{A}_{H,K}$ in general, even if $H$ and $K$ share no common non-trivial direct factor.

\begin{example}\label{alpha0} There exist two torsion-free abelian groups $H$ and $K$ with no common non-trivial direct summand such that $Aut(H\oplus K)$ is not contained in $\mathcal{A}_{H,K}$.
\end{example}
\begin{proof}
Let $H$ and $K$ be the groups constructed in Example \ref{Ex1}, where, in this case, $s=2$ and $r=3$. From the first paragraph of Exercise \ref{Ex1}, we get that $H$ and $K$ are incompatible and that in particular they share no common non-trivial direct summand. Let $\beta:x\in K\mapsto sx\in H$, let $\gamma$ be the immersion of $H$ in $K$, let $\delta$ be the endomorphism of $K$ defined by the multiplication by $r$ and take into account the following element of $\End(H\oplus K)$: $$\varphi=\left(\begin{array}{cc}1 &\beta \\ \gamma & \delta\end{array}\right).$$ Since for instance $x$ has no infinite $r$-height in $K$, $\delta$ is not surjective and hence $\varphi$ is not an element of $\mathcal{A}_{H,K}$. On the other hand, take into account $\det_K\varphi=\delta-\gamma\beta$. If $k$ is any element of $K$, then $\delta-\gamma\beta(k)=rk-sk=k$ and hence $\det_K\varphi$ is the identity. It now follows from Theorem \ref{propinv} that $\varphi$ is an automorphism of $H\times K$.
\end{proof}

In the previous examples, $H$ and $K$ were abelian. We are now exhibiting two incompatible non-abelian groups showing that, for instance, the hypothesis on normal subgroups in Theorem \ref{ThCompAut} cannot be removed also when the groups are not abelian.

\begin{example}\label{Ex3} There exist two infinite groups $H$ and $K$ which are compatible, centrally incompatible and share no non-trivial direct factor and there exist two automorphisms $\varphi$ and $\varphi'$ of $Aut(H\times K)$ such that $\pi_H\varphi\iota_H$ and $\pi_K\varphi'\iota_K$ are not surjective. In particular, $Aut(H\times K)$ is not contained in $\mathcal{A}_{H,K}$.
\end{example}
\begin{proof}
Let $H=\gen{x}$ be an infinite cyclic group and let $$K=\gen{y,z\mid z^3=1,z^y=z^{-1}},$$ namely a group generated by an infinite cyclic group acting as the inversion on a cyclic group of order $3$. Here we employ multiplicative notation. Let $\alpha$, $\beta$ and $\gamma$ be the homomorphisms defined by the following rules $$\alpha:x\in H\mapsto x^3\in H,\quad\beta:\begin{cases}y\in K\mapsto x\in H \\z\in K\mapsto 1\in H\end{cases} \text{and} \quad\gamma:x\in H\mapsto y^2\in K$$ and take $$\varphi=\left(\begin{array}{cc}\alpha &\beta \\ \gamma & 1\end{array}\right)$$
to be an endomorphism of $H\times K$. Since $\det_H\varphi=1$, $\varphi$ is invertible. Moreover, with the help of Proposition \ref{invgen}, we may easily find the inverse of $\varphi$, namely $$\left(\begin{array}{cc}1 &-\beta \\ -\gamma & 1+\gamma\beta\end{array}\right),$$ so that they both belong to $\aut(H\times K)$. However, $\pi_H\varphi\iota_H$ and $\pi_K\varphi^{-1}\iota_K$ are not surjective, so that neither $\varphi$ nor $\varphi^{-1}$ belong to $\mathcal{A}_{H,K}$.

Now, if we take $\psi\in\Hom(H,K)$ to be the homomorphism defined by $\psi:x\in H\mapsto y\in K$ and take $\chi\in\Hom(K,H)$ to be the only homomorphism which sends $y$ to $x$, then we get that $H$ and $K$ are compatible. So, the last thing to prove is that $H$ and $K$ are centrally incompatible. To this aim, take $\sigma$ in $\Hom(H,Z(K))\setminus\{0\}$, $\tau$ in $\Hom(K,Z(H))\setminus\{0\}$ and let $k\in K$ such that $\sigma\tau(k)=k$. By the structure of $H$ and $K$, we have that there are two non-zero integers $m$ and $n$ such that $\sigma$ and $\tau$ may be defined by the following correspondences: $\sigma:x\in H\mapsto y^{2m}\in K$ and $\beta:\begin{cases}y\in K\mapsto x^{n}\in H \\z\in K\mapsto 1\in H\end{cases}$. It is now clear that $k=k^{mn}$ and also that $k\in\gen{y}$. These conditions hold true only in case $k$ is trivial. Hence $H$ and $K$ are incompatible.
\end{proof}

Starting from this example, one can substitute $H$ with the $H$ in Example \ref{alpha0} and substitute the copy of $\mathbb{Z}$ acting on $\gen{z}$ with the $K$ in Example \ref{alpha0} acting in the same fashion on $\gen{z}$. Then it is not difficult to prove the following

\begin{example}\label{Ex4} There exist two infinite incompatible groups $H$ and $K$ and two automorphisms $\varphi$ and $\varphi'$ of $Aut(H\times K)$ such that $\pi_H\varphi\iota_H$ and $\pi_K\varphi'\iota_K$ are not surjective. In particular, $Aut(H\times K)$ is not contained in $\mathcal{A}_{H,K}$.
\end{example}

\subsection{More on stem groups}

Here we provide an analogue of Example \ref{ExTotInc} for stem groups. In Theorem \ref{stem} we proved that if either $H$ or $K$ is a stem group, then $\aut(H\times K)\simeq\mathcal{A}_{H,K}$. Moreover, it is an easy remark that in this case $H$ and $K$ are centrally incompatible. If in the same hypothesis $H$ and $K$ were totally incompatible, than an application of Proposition \ref{htotincompatible} would have implied the thesis of Theorem \ref{stem} immediately. However, this is not the case, even if both $H$ and $K$ are stem groups.

\begin{example}\label{Ex5} There exist two infinite stem groups $H$ and $K$ which are incompatible but not totally incompatible.
\end{example}
\begin{proof}
Let $p,q,r$ and $s$ be different positive primes and let $X$ and $Y$ be the subgroups of $\mathbb{Q}\oplus\mathbb{Q}$ generated respectively by the subsets $$\{p^lx,q^my,pqr^n(x+y)\mid l,m,n\in\mathbb{Z}\}\quad\text{and}\quad\{p^lx,q^my,r^n(x+y)\mid l,m,n\in\mathbb{Z}\}$$ for a basis $\{x,y\}$ of $\mathbb{Q}\times\mathbb{Q}$. Let $Q_1=X\oplus X$ and $Q_2=Y\oplus Y$ and let $D$ be the direct product of two Prüfer groups of type $p^\infty$, and $q^\infty$, respectively. Then it is possible to find, for instance making use of the five-term homology sequence, two epimorphisms $\varphi_1:M(Q_1)\twoheadrightarrow D$ and $\varphi_2:M(Q_2)\twoheadrightarrow D$, where $M(Q_1)$ and $M(Q_2)$ are the Schur multipliers of $Q_1$ and $Q_2$, respectively. By the Universal Coefficients Theorem we can now take $H$ and $K$ to be the central extensions of $Q_1$ by $D$ induced by $\varphi_1$ and of $Q_2$ by $D$ induced by $\varphi_2$, respectively. Since $\varphi_1$ and $\varphi_2$ are surjective, $H'\leq Z(H)$ and $K'\leq Z(K)$. Moreover, no element of infinite order of $H$ or $K$ is central and hence $H'=Z(H)$ and $K'=Z(K)$. In particular, $H$ and $K$ are stem groups. For the sake of simplicity, we assume that $Z(H)=D=Z(K)$ and that $H/D=Q_1$ and $K/D=Q_2$.

%Queste due cose sono inutili, se davvero dimostro l'incompatibilità.
%"From the fact that $X$ and $Y$ are indecomposable (which can be proved for instance by means of \cite[Lemma 88.3]{fuch2}) and that $H'=Z(H)$ and $K'=Z(K)$, it is not difficult to see that $H$ and $K$ are indecomposable."
%"Questo va messo dopo, forse giustificandolo con la mancanza di isomorfismi tra Q_1 e Q_2... Moreover, as it is shown in Example \ref{Ex1}, $X$ is not isomorphic with $Y$, so that $H/D$ is not isomorphic with $K/D$ and hence $H$ and $K$ are not isomorphic. In particular, $H$ and $K$ share no common non-trivial direct factor."

For $j\in\{p,q\}$, let now $D_j$ be the $j$-component of $D$ and let $\{d_{j,i}\mid i\in\mathbb{N}\setminus\{0\}\}$ be a set of generators of $D_j$, where $d_{j,i}^p=d_{j,i-1}$ for every $i\in\mathbb{N}\setminus\{0,1\}$ and $d_{j,1}^j=1$. Moreover, define $\psi_1:(a,b)\in Q_1\mapsto a\wedge b\in M(Q_1)$ and  $\psi_2:(a,b)\in Q_2\mapsto a\wedge b\in M(Q_1)$, where we regard $M(Q_i)$ as the exterior square of $Q_i$ for $i\in\{1,2\}$. Without loss of generality, we may assume that $$\varphi_1\psi_1(1/p^ix,1/p^ix)=d_{p,i},\quad\varphi_1\psi_1(1/q^iy,1/q^iy)=d_{q,i}\quad\text{and}\quad$$
$$\varphi_1\psi_1(pq/r^i(x+y),pq/r^i(x+y))=1$$ and also  $$\varphi_2\psi_2(1/p^ix,1/p^ix)=d_{p,i},\quad\varphi_2\psi_2(1/q^iy,1/q^iy)=d_{q,i}\quad\text{and}\quad$$
$$\varphi_2\psi_2(1/r^i(x+y),1/r^i(x+y))=1$$ for each $i\in\mathbb{N}\setminus\{0\}$.
These are in fact some commutator relations in $H$ and $K$ and the other ones can be defined accordingly. Looking at the two set of relations, we may regard $H$ as a subgroup of $K$ and take $\gamma$ to be the immersion of the former into the latter. Moreover, let $\overline{\beta}:z\in Q_2\mapsto pqz\in Q_1$ and $\tilde{\beta}:(a,b)\in D\mapsto (a^{p^2q^2},b^{p^2q^2})\in D$. If we look at $\varphi_1$ and $\varphi_2$ as the cohomology classes of the extensions
$$\begin{tikzcd}
D \arrow[tail]{r} & H\arrow[two heads]{r} & Q_1
\end{tikzcd}\quad\text{and}\quad\begin{tikzcd}
D \arrow[tail]{r} & K\arrow[two heads]{r} & Q_2,
\end{tikzcd}$$
respectively, it is not difficult to deduce from the relations that $\tilde{\beta}_*(\varphi_2)=\overline{\beta}^*(\varphi_1)$, so that it follows form Proposition II.4.3 in \cite{Stamm} that we may find a homomorphism $\beta:K\rightarrow H$ inducing $\tilde{\beta}$ on $D$ and $\overline{\beta}$ on $K/D$. If we now take into account the normal endomorphism $\gamma\beta$ of $K$, we can see that it always sends elements of infinite order to elements of infinite order. In particular, $H$ and $K$ are not totally incompatible.

\smallskip

Finally, we want to prove that $H$ and $K$ are incompatible. To this aim, let $\psi$ and $\chi$ be elements of $\Hom(K,H)$ and $\Hom(H,K)$, respectively, and let $k$ be an element of $K$ such that $\chi\psi(k)=k$. As each direct factor of $Q_1$ centralizes $D$ and $D$ is divisible, we can find two subgroups $X_1$ and $X_2$ of $H$ such that $X_1\simeq X_2\simeq X$ and $H=(X_1\times D)(X_2\times D)$. The same can be clearly done with $K$, so that we find two subgroups $Y_1$ and $Y_2$ such that $Y_1\simeq Y_2\simeq Y$ and $K=(Y_1\times D)(Y_2\times D)$. For $i\in\{1,2\}$, write $X_i$ and $Y_i$ as the subgroups of $\mathbb{Q}\times\mathbb{Q}$ generated respectively by the following subsets $$\{p^lx_i,q^my_i,pqr^n(x_i+y_i)\mid l,m,n\in\mathbb{Z}\}\quad\text{and}\quad \{p^lx_i,q^my_i,r^n(x_i+y_i)\mid l,m,n\in\mathbb{Z}\}.$$
Here $X_i$ is a subgroup of $Y_i$, and in fact we will regard $Q_2$ as a subgroup of $Q_1$ for the sake of simplicity. Moreover, if $W$ is one between $X$ and $Y$, let $W_{i,j}$ be the subgroup of $W_i$ generated by every element of infinite $j$-height for $i\in\{1,2\}$ and $j\in\{p,q,r\}$. As $X_{1,j}X_{2,j}D/D$ is the subgroup of $Q_1$ containing all and only the elements of infinite $j$-height, then every element of infinite $j$-height of $H$ is contained in $X_{1,j}X_{2,j}D$. Similarly, every element of infinite $j$-height of $K$ is contained in $Y_{1,j}Y_{2,j}D$.

Let $\overline{\psi}$ be the homomorphism from $Q_2$ to $Q_1$ induced by $\psi$ and let $\overline{\chi}$ be the homomorphism $Q_1$ to $Q_2$ induced by $\chi$. Since $\psi$ preserves heights, we may find the integers $a,a',a'',b,b',b'',c,c',c''$ and positive integers $l,l'$, $l'',m,m',m'',n,n'$ and $n''$ such that we get the following equalities modulo $D$ $$\overline{\psi}(x_1)=p^al(p^{a'}l'x_1+p^{a''}l''x_2),\quad\overline{\psi}(y_1)=q^bm(q^{b'}m'y_1+q^{b''}m''y_2)\quad\text{and}\quad$$ $$\overline{\psi}(x_1+y_1)=pqr^cn(r^{c'}n'(x_1+y_1)+r^{c''}n''(x_2+y_2)).$$ This yields that there are two integers $s$ and $t$ such that (again modulo $D$) $$\overline{\psi}(x_1)=pq(sx_1+tx_2),\quad\overline{\psi}(y_1)=pq(sy_1+ty_2)\quad\text{and}\quad$$ $$\overline{\psi}(x_1+y_1)=pq(s(x_1+y_1)+t(x_2+y_2)).$$ In the same way, we find the integers $u$ and $v$ such that the following equalities modulo $D$ hold $$\overline{\psi}(x_2)=pq(ux_1+vx_2),\quad\overline{\psi}(y_2)=pq(uy_1+vy_2)\quad\text{and}\quad$$ $$\overline{\psi}(x_2+y_2)=pq(u(x_1+y_1)+v(x_2+y_2)).$$ In particular, $\im\overline{\psi}$ is always contained in $pqQ_2$, when we see it as a subgroup of $Q_1$. Let now $h_j$ be the $j$-height function over $K$ for $j\in\{p,q,r\}$. As in Example \ref{Ex1}, since $Q_1$ is totally invariant in $Q_2$, we have that for every non-trivial element $k_j$ of $Y_{1,j}Y_{2,j}D/D$, $h_j(k_j)<h_j(\overline{\chi}\overline{\psi}(k_j))$. From this we get that $k$ has to be a torsion element. For if we write $kD=k_pk_qk_r$, where $k_p\in Y_{1,p}Y_{2,p}D/D$, $k_q\in Y_{1,q}Y_{2,q}D/D$ and $k_r\in Y_{1,r}Y_{2,r}D/D$, then $\overline{\chi}\overline{\psi}(k_pk_qk_r)=k_pk_qk_r$, so that $h_q(k_p)=\overline{\chi}\overline{\psi}(k_p)$, $h_p(k_q)=\overline{\chi}\overline{\psi}(k_q)$ and $h_p(k_r)=\overline{\chi}\overline{\psi}(k_r)$ and hence $k_pk_qk_r=0$. Let now $d$ be any element of $D$. Then we may find elements $k_{1,p}$, $k_{2,p}$, $k_{1,q}$ and $k_{2,q}$ of $Y_{1,p}$, $Y_{2,p}$, $Y_{1,q}$ and $Y_{2,q}$, respectively, such that $d=[k_{1,p},k_{2,p}][k_{1,q},k_{2,q}]$. However, $\overline{\chi}\overline{\psi}$ maps $Q_2$ to $pqQ_2$ and hence $\chi\psi(d)=\chi\psi([k_{1,p},k_{2,p}])\chi\psi([k_{1,q},k_{2,q}])$, which belongs to $\gen{d^{p^2q^2}}<\gen{d}$. Therefore, $k=1$ and $H$ and $K$ are incompatible.
\end{proof}

\section{Future work}\label{SecFut}
The results presented in this paper can be seen as a new starting point for a study of automorphisms of group extensions, beginning from the simplest ones, namely direct products. From this point of view, an answer to the following problem can be of great interest.

\begin{prob}
Define a determinant for group extensions.
\end{prob}

If possible, this could be made by means of homological tools developed for instance in \cite{Stamm}. Going easy, the following may be the next step of enquiry (on this, see for instance \cite{Curr} and the more recent \cite{CCCo}).

\begin{prob}
Define a determinant for semidirect products of groups.
\end{prob}

These are general questions, which can be tackled independently from the study of automorphisms of direct products of groups, which is still ongoing. Indeed, there is work to do also for direct products of two indecomposable groups. In this case, a possible way to study when direct products of groups have "good" automorphism group can be achieved from a deeper study of incompatible pairs of various kind. This problems can also have an interest on their own. For instance, a classification of pairs of totally incompatible abelian groups can show to what extent the determinants can be used for the whole automorphism group.

\begin{prob}
Study (centrally/totally) incompatible group pairs.
\end{prob}

Among other things, we believe something can be easily said for indecompos\-able torsion-free abelian groups of finite rank. On the other hand, as far as central incompatibility is concerned, stem groups suggest the following definition. Let $n$ be a positive integer. We say that two groups $H$ and $K$ are \textit{(centrally) totally incompatible of length} $n$ or, equivalently, that the pair $(H,K)$ is \textit{(centrally) totally incompatible of length} $n$, if $n$ is the least positive integer $n$ such that, for every $\sigma\in\Hom(H,K)$ ($\in\Hom(H,Z(K))$, respectively), $\tau\in\Hom(K,H)$ ($\in\Hom(H,Z(K))$, respectively), $h\in H$ and $k\in K$, if $\sigma\tau$ and $\tau\sigma$ are normal, then either
\begin{itemize}
\item[(1)] $(\sigma\tau)^n(k)=1$ or
\item[(2)] $(\tau\sigma)^n(h)=1$.
\end{itemize}
As pointed out in Subsection \ref{SubStem}, if one between $H$ and $K$ is a stem group, then the pair $(H,K)$ is centrally totally incompatible of length at most $2$, while if both $H$ and $K$ are stem groups, then the pair $(H,K)$ is centrally totally incompatible of length at most $1$. Clearly, if $(H,K)$ is totally incompatible of fixed length, it is also totally incompatible and hence the group $H\times K$ has a well-described automorphism group.  From this perspective, a study of this pairs can be interesting.

\begin{prob}
Study (centrally) totally incompatible group pairs of length $n$ for a fixed positive integer $n$.
\end{prob}

These problems can also be rephrased for $m$-tuples of groups. This task should not be technically very different from the case $m=2$.

\begin{prob}
Develop explicitly an incompatibility theory for direct products of arbitrari\-ly many groups.
\end{prob}

Independently from this, it would be useful to inspect the behaviour of automorphisms of direct products of a finite number of groups. A taste of this has already been given in Subsection \ref{SubSketch}, as for the determinants, but the work is far from being concluded. For example, one could ask for an explicit formula for the inverse of an automorphism, just as in Proposition \ref{invgen} as for the case $n=2$.

\begin{prob}
Let $H_1,\ldots, H_n$ be groups and let $\varphi$ be an automorphism of $H_1,\ldots, H_n$ for which a determinant can be defined. Express $\varphi^{-1}$ in terms of the components of $\varphi$ as an element of $\mathcal{M}_{H_1,\ldots, H_n}$.
\end{prob}

Moreover, the general case, in particular the case in which the factors of a direct product of groups do not satisfy some incompatibility hypothesis, is far from being completed. Starting from \cite{Bid}, one can ask which of those results hold for not necessarily finite groups. Then we state the following, very vague problem.

\begin{prob}
Develop the theory of automorphisms of direct products of finitely many groups.
\end{prob}

Finally, as suggested in Subsection \ref{GAP}, an implementation of the determinants and a further inspection of their computational advantages may prove useful for any computer algebra system, even in dealing with some infinite structure.

\begin{prob}
Implement different algorithms using the determinants and study their computational costs.
\end{prob}

\addcontentsline{toc}{chapter}{\hspace{13pt} Bibliography}

\bigskip

\begin{center}
\rule{0.5\textwidth}{.4pt}
\end{center}

\bigskip

\bigskip

MATTIA BRESCIA, Dipartimento di Matematica e Applicazioni,

Università degli Studi di Napoli Federico II

Corso Umberto I, 40, Napoli, Italy

e-mail: mattia.brescia@unina.it

\end{document}